\documentclass[11pt,centertags,reqno,twoside]{amsart}
\usepackage{amssymb}
\usepackage{mathptmx}
\usepackage[mathcal]{euscript}

\usepackage[english]{babel}

\DeclareSymbolFont{SY}{U}{psy}{m}{n}
\DeclareMathSymbol{\emptyset}{\mathord}{SY}{'306}

\allowdisplaybreaks

\numberwithin{equation}{section}

\newtheorem{theorem}{Theorem}
\newtheorem{corollary}[theorem]{Corollary}
\newtheorem{lemma}[theorem]{Lemma}

\theoremstyle{definition}
\newtheorem{definition}[theorem]{Definition}
\theoremstyle{remark}
\newtheorem{remark}[theorem]{Remark}

\begin{document}

\title[Spectral analysis of an even order differential operator]
{Spectral analysis of an even order differential operator}

\author[D. M. Polyakov]
{Dmitry M. Polyakov}

\address{Dmitry M. Polyakov\newline\hspace*{9mm} Southern Mathematical Institute, Vladikavkaz Scientific Center of RAS,
Vladikavkaz, Russia}
\email{DmitryPolyakow@mail.ru}

.

\begin{abstract}
Using the method of similar operators we study an even order differential operator with periodic, semiperiodic,
and Dirichlet boundary conditions. We obtain asymptotic formulas for eigenvalues of this operator and estimates
for its spectral decompositions and spectral projections. We also establish the asymptotic behavior of the corresponding
analytic semigroup of operators.
\end{abstract}

\maketitle

\section{Introduction}\label{intro}

We consider an even order differential operator of the following type
\[
L_{bc}=(-1)^k\frac{d^{2k}}{dx^{2k}}-q, \quad k>1,
\]
where $q$ is a potential and $bc$ stands for boundary conditions. One of the main problems of the spectral analysis of
this operator is to establish sharp enough asymptotic formulas for its eigenvalues. As a rule, it is required the continuity
of the potential $q$. But last time the case of nonsmooth potentials is considered actively by many authors.
Let us give a brief survey of the most important achievements for this case.

For $k=1$ the operators $L_{bc}$ with various boundary conditions are the Hill ones.
They were studied in details in \cite{Marchenko}, \cite{Savchuk_Shkalikov} (see also the references therein).
In \cite[Ch.~II, Theorem~2]{Naimark} it was established the asymptotic behavior of eigenvalues for this
operator with regular and strongly regular boundary conditions.
In \cite{Badanin_Korotyaev_IMRN_2011} A. Badanin and E. Korotyaev carried out spectral
analysis of even order differential operator in $L_2(\mathbb{R})$ with periodic potential from the space $L_1(\mathbb{T})$,
where $\mathbb{T}=\mathbb{R}\setminus\mathbb{Z}$.
They obtained many interesting results for this operator concerning description of its spectrum, asymptotic behavior of
eigenvalues, and spectral gaps. In the case $k=2$ similar results were obtained in \cite{Badanin_Korotyaev_IMRN_2005}
for the operator $L_{bc}$ with periodic boundary conditions and potential $q$ from the real
space $L_1(\mathbb{T})$. Fourth order differential operator of a general form was considered in papers
\cite{Badanin_Korotyaev_AA,PolyakovAA,<Polyakov1>,Kerimov}.

Asymptotic behavior of eigenvalues for the operator $L_{bc}$ in the negative Sobolev space $H^{-m}[-1, 1]$ with periodic and
semiperiodic boundary conditions and potential $q$ from $H^{-m}[-1, 1]$ were given in
\cite{<Molyboga_Mikhailets1>,<Molyboga_Mikhailets2>,<Molyboga>}.

In \cite{Veliev} O.A. Veliev considered the nonself-adjoint ordinary differential operators with periodic and semiperiodic
boundary conditions and summable complex-valued coefficients. He obtained asymptotic formulas for their eigenvalues and
eigenfunctions. He also established necessary and sufficient conditions on the coefficient under which the root functions
of these operators form a Riesz basis in $L_2[0, 1]$.

In \cite[Theorem~1]{Ahmerova} and \cite{<Menken>} it was described the asymptotic behavior of eigenvalues for an even
order differential operator with the Dirichlet boundary conditions and for the operator $L_{bc}$ in the case $k=2$ and periodic
boundary conditions, respectively.

The aim of this paper is to give a detailed spectral analysis for an even order differential operator with a nonsmooth
complex-valued potential and various boundary conditions. To do this, we use a new approach which is based on the method of
similar operators developed in the general setting in \cite{<1983>,<1987>,<2011>,Bask_Pol_2017}.
This method allows us to reduce the study to an associated operator which is simpler than the initial operator $L_{bc}$.

The paper is organized as follows. In Section~\ref{sec1} we state the main results.
In Section~\ref{sec2} we recall some basic notions for the method of similar operators and study an abstract operator
having properties similar to those of the operator $L_{bc}$ with periodic, semiperiodic, and Dirichlet boundary conditions.
In Section~\ref{sec3} the results of Section~\ref{sec2} are applied to the operator $L_{bc}$.
The main results are proven in the final Section~\ref{sec4}.

\section{Preliminaries and main results}\label{sec1}

Let $L_2[0, \omega]$ be the Hilbert space of square summable complex functions on $[0, \omega]$, $\omega>0$, with inner
product $(x, y)=\frac{1}{\omega}\int_0^\omega x(\tau)\overline{y(\tau)}\,d\tau$, $x$, $y\in L_2[0, \omega]$.
By $H^{2k}[0, \omega]$, $k>1$, we denote the Sobolev space $\{y\in L_2[0, \omega]\to\mathbb{C}:
y$ is $2k$-times differentiable, $y^{(2k-1)}$ is absolutely continuous, $y^{(2k)}\in L_2[0, \omega]\}$.

We consider the operators $L_{bc}: D(L_{bc})\subset L_2[0, \omega]\to L_2[0, \omega]$ defined by the
following differential expression
\[
l(y)= (-1)^k y^{(2k)} - qy \quad \text{with} \quad k>1 \quad \text{and} \quad q\in L_2[0, \omega]
\]
and the boundary conditions $bc$ of the following types:
\begin{enumerate}
\item[(a)] periodic $bc=per$: $y^{(j)}(0)=y^{(j)}(\omega)$, $j=0, 1, \dots, 2k-1$;

\item[(b)] semiperiodic $bc=ap$: $y^{(j)}(0)=-y^{(j)}(\omega)$, $j=0, 1, \dots, 2k-1$;

\item[(c)] Dirichlet $bc=dir$: $y(0)=y''(0)=\dots=y^{(2k-2)}(0)=0$, $y(\omega)=y''(\omega)=\dots=y^{(2k-2)}(\omega)=0$.
\end{enumerate}
Thus, $D(L_{bc})=\{y\in H^{2k}[0, \omega]: y\; \text{satisfies conditions}\; bc\}$. The operators $L_{bc}$ with the boundary
conditions (a), (b), and (c) will be denoted by $L_{per}$, $L_{ap}$, and $L_{dir}$, respectively.

The operator $L_{bc}$ may be applied to investigate vibrations of beams, plates and shells (see~\cite{Collatz}). For example
(see \cite[Sec.~1.14]{Novozhilov}), the Vlasov model of the bending cylindrical shells
gives vibration equations of the form $y^{(8)}+by=\lambda y$, where $b$ is a potential and $\lambda$ is a spectral parameter.

To state our main results, we need some notation.
The operator $L_{bc}$ can be represented in the form $L_{bc}=L_{bc}^0-Q$, where
$L_{bc}^0: D(L_{bc}^0)=D(L_{bc})\subset L_2[0, \omega]\to L_2[0, \omega]$, $L_{bc}^0y=(-1)^ky^{(2k)}$,
$k>1$, is an unperturbed operator and $Q$ is the operator of multiplication by the potential $q$. As is well known,
$L_{bc}^0$ is a self-adjoint operator with discrete spectrum.

For $bc\in\{per, ap\}$ the operators $L_{bc}$ and $L_{bc}^0$ will be denoted by $L_\theta$ and $L_\theta^0$,
where $\theta=0$ and $\theta=1$ stands for $bc=per$ and $bc=ap$, respectively.

Now we describe the spectrum $\sigma(L_{bc}^0)$ and the eigenfunctions of the operators $L_{per}^0$, $L_{ap}^0$, $L_{dir}^0$:

(a), (b): $\sigma(L_\theta^0)=\{\lambda_n, n\in\mathbb{Z}_+=\mathbb{N}\cup\{0\}\}$, where
$\lambda_n=\pi^{2k}(2n+\theta)^{2k}/\omega^{2k}$, $k>1$; the corresponding eigenfunctions are
$e_n(t)=\mathrm{e}^{-i\pi(2n+\theta)t/\omega}$, $t\in[0, \omega]$. They form an orthonormal
basis in $L_2[0, \omega]$;

(c): $\sigma(L_{dir}^0)=\{\lambda_{n, dir}, n\in\mathbb{N}\}$, where $\lambda_{n, dir}=\pi^{2k}n^{2k}/\omega^{2k}$,
$k>1$; the corresponding eigenfunctions have the form $e_{n, dir}(t)=\sqrt{2}\sin(\pi nt/\omega)$, $t\in[0, \omega]$.

Note that the operator $L_{per}^0$ has double eigenvalues (except the eigenvalue $\lambda_0=0$). The operator $L_{ap}^0$ has
double eigenvalues and the operator $L_{dir}^0$ has simple eigenvalues.

Since the potential $q$ belongs to $L_2[0, \omega]$, $q(t)=\sum_{l\in\mathbb{Z}}
q_le^{i2\pi lt/\omega}$. For the operator $L_{dir}^0$ we will use the following representation for the potential:
$q(t)=\sqrt{2}\sum_{l=1}^\infty \widetilde{q}_l\cos(\pi lt/\omega)$.

Let us agree to use the symbol $C$ for all positive constants, perhaps different one from another.

Now we state the main results of this paper.
\begin{theorem}\label{thasympt0and1H}
The operators $L_{bc}$, $bc\in\{per, ap\}$, have discrete spectrums and there exists $m\in\mathbb{Z}_+$
such that the spectrum $\sigma(L_{bc})$ has the form
\begin{equation}\label{sigmaLthetaH}
\sigma(L_{bc})=\sigma_{(m)}\cup \{\widetilde{\lambda}_n^\mp, n\geq m+1\},
\end{equation}
where $\sigma_{(m)}$ is a finite set with a number of points not exceeding $2m+1$. The eigenvalues
$\widetilde{\lambda}_n^\mp$, $n\geq m+1$, have the following asymptotic representation
\begin{flalign}\label{lambdatheta0and1*H}
\widetilde{\lambda}_n^\mp &= \Big(\frac{\pi(2n+\theta)}{\omega}\Big)^{2k}-q_0 -
\frac{2\omega^{2k}}{\pi^{2k}}\sum\limits_{\substack{j=1 \\ j\ne n}}^\infty
\frac{q_{n-j}q_{j-n}}{(2j+\theta)^{2k}-(2n+\theta)^{2k}} \nonumber\\
&\mp \Big(q_{-2n-\theta} + \frac{\omega^{2k}}{\pi^{2k}}\sum\limits_{\substack{j\in\mathbb{Z} \\ j\ne n,\; j\ne -n-\theta}}
\frac{q_{-n-j-\theta}q_{j-n}}{(2j+\theta)^{2k}-(2n+\theta)^{2k}}\Big)^\frac{1}{2}\nonumber\\
&\cdot\Big(q_{2n+\theta} + \frac{\omega^{2k}}{\pi^{2k}}\sum\limits_{\substack{j\in\mathbb{Z} \\ j\ne n,\; j\ne -n-\theta}}
\frac{q_{n+j+\theta}q_{n-j}}{(2j+\theta)^{2k}-(2n+\theta)^{2k}}\Big)^\frac{1}{2} + \xi_{bc}(n).
\end{flalign}
Here the sequence $\xi_{bc}: m+\mathbb{N}\to (0, \infty)$ satisfies the estimate:
$|\xi_{bc}(n)|\leq C\alpha_n/n^{4k-3}$, where $(\alpha_n)$ is a square summable sequence.
\end{theorem}
\begin{corollary}
Let $q$ be a real-valued function from $L_2[0, \omega]$. Then the eigenvalues
$\widetilde{\lambda}_n^\mp$, $n\geq m+1$, have the form
\begin{flalign}\label{lambdatheta0and1*H}
\widetilde{\lambda}_n^\mp &= \Big(\frac{\pi(2n+\theta)}{\omega}\Big)^{2k}-q_0 -
\frac{2\omega^{2k}}{\pi^{2k}}\sum\limits_{\substack{j=1 \\ j\ne n}}^\infty
\frac{q_{n-j}q_{j-n}}{(2j+\theta)^{2k}-(2n+\theta)^{2k}}\pm \nonumber\\
&\pm \Big(q_{2n+\theta} + \frac{\omega^{2k}}{\pi^{2k}}\sum\limits_{\substack{j\in\mathbb{Z} \\ j\ne n,\; j\ne -n-\theta}}
\frac{q_{n+j+\theta}q_{n-j}}{(2j+\theta)^{2k}-(2n+\theta)^{2k}}\Big) + \xi_{bc}(n).
\end{flalign}
Here the sequence $\xi_{bc}: m+\mathbb{N}\to (0, \infty)$ satisfies the estimate:
$|\xi_{bc}(n)|\leq C\alpha_n/n^{4k-3}$, where $(\alpha_n)$ is a square summable sequence.
\end{corollary}
\begin{theorem}\label{thasympvarH}
Let $q$ be a function of bounded variation; then the asymptotic representation (\ref{lambdatheta0and1*H}) holds
and the sequence $\xi_{bc}$ satisfies the estimate
$|\xi_{bc}(n)|\leq C/n^{4k-2}$, $n\geq m+1$, where $m$ is defined in Theorem~\ref{thasympt0and1H}.
\end{theorem}

Theorem~\ref{thasympvarH} also holds if the potential $q$ is a smooth function and improves a similar result of
H.~Menken \cite[Theorem 3.1]{<Menken>}.
\begin{theorem}\label{thasympt(0,1)H}
The operator $L_{dir}$ has a discrete spectrum. There exists $m\in\mathbb{N}$ such that its spectrum
is represented in the form (\ref{sigmaLthetaH}), where $\sigma_{(m)}$ is a finite set with a number of points not exceeding $m$
and $\sigma_n=\{\widetilde{\lambda}_{n, dir}\}$, $n\geq m+1$.
For the eigenvalues $\widetilde{\lambda}_{n, dir}$, $n\geq m+1$, we have the following representation
\begin{flalign}\label{lambdantheta01H}
\widetilde{\lambda}_{n, dir}=\Big(\frac{\pi n}{\omega}\Big)^{2k} &-  \frac{1}{\omega}\int_0^\omega q(t)\,dt +
\frac{1}{\omega}\int_0^\omega q(t)\cos\frac{2\pi n}{\omega}t\,dt \nonumber\\
&- \frac{\omega^{2k}}{2\pi^{2k}}
\sum\limits_{\substack{j=1 \\ j\ne n}}^\infty\frac{(\widetilde{q}_{|n-j|}-\widetilde{q}_{n+j})^2}{j^{2k}-n^{2k}} +
\eta_{dir}(n), \quad n\geq m+1,
\end{flalign}
where the sequence $\eta_{dir}:\{n\in\mathbb{N}\;|\;n\geq m+1\} \to (0, \infty)$ satisfies the estimate:
$|\eta_{dir}(n)|\leq C\beta_n/n^{4k-3}$ with some  square summable sequence $(\beta_n)$.

If $q$ is a function of bounded variation (or smooth function), then we have
\begin{flalign}\label{qdirbondvar}
\widetilde{\lambda}_{n, dir} &=\Big(\frac{\pi n}{\omega}\Big)^{2k} - \frac{1}{\omega}\int_0^\omega q(t)\,dt +
\frac{1}{\omega}\int_0^\omega q(t)\cos\frac{2\pi n}{\omega}t\,dt \nonumber\\
&- \frac{\omega^{2k}}{2\pi^{2k}}\sum\limits_{\substack{j=1 \\ j\ne n}}^\infty
\frac{(\widetilde{q}_{|n-j|}-\widetilde{q}_{n+j})^2}{j^{2k}-n^{2k}} +
\mathcal{O}(n^{-4k+2}), \quad n\geq m+1.
\end{flalign}
\end{theorem}
From Theorem~\ref{thasympt(0,1)H} it follows directly
\begin{corollary}\label{corSpectr}
The operator $L_{dir}$ is a spectral operator (in the Dunford sense) (see \cite{<Danford>}).
\end{corollary}

To state the next result, we need some additional notation.

By $\mathbb{P}_n$, $n\in\mathbb{Z}_+$, denote the Riesz projections constructed for the sets
$\{\pi^{2k}(2n+\theta)^{2k}/\omega^{2k}\}$. Next, by $P_{n, dir}$, $n\in\mathbb{N}$, denote the Riesz projections
constructed for $\{\pi^{2k}n^{2k}/\omega^{2k}\}$. Hence, $L_{bc}^0\mathbb{P}_n=\lambda_n\mathbb{P}_n$,
$n\in\mathbb{Z}_+$, for $bc\in\{per, ap\}$, and $L_{dir}^0P_{n, dir}=\lambda_{n, dir}P_{n, dir}$, $n\in\mathbb{N}$. For all
$x\in L_2[0, \omega]$ these projections have the following form:
\begin{flalign}
\textup{(a), (b):}\; &\mathbb{P}_nx=P_{-n-\theta}x+P_nx=(x, e_{-n-\theta})e_{-n-\theta}+(x, e_n)e_n, \quad n\in\mathbb{Z}_+;
\nonumber \\
&\mathbb{P}_0x=P_0x=(x, e_0)e_0\quad \text{for}\quad \theta=0\quad \text{and}\quad n=0;\label{pndef}\\
\textup{(c):}\; &P_{n, dir}x=(x, e_{n, dir})e_{n, dir}, \quad n\in\mathbb{N}.\label{pndirdef}
\end{flalign}

Let $m$ be as in Theorem~\ref{thasympt0and1H} or Theorem~\ref{thasympt(0,1)H} (for $bc\in\{per, ap\}$ or $bc=dir$, respectively).
We denote by $\mathbb{P}_{(m)}$ the projection $\sum_{j=0}^m\mathbb{P}_j$ for $bc\in\{per, ap\}$ and by $P_{(m)}$
the projection $\sum_{j=1}^mP_{j, dir}$ for $bc=dir$.

Let $bc\in\{per, ap\}$ and $\Omega\subset\mathbb{Z}_+\setminus\{0, \dots, m\}$. For the set
$\Delta=\Delta(\Omega)=\{\lambda_n, n\in\Omega\}$ the Riesz projection $P(\Delta, L_{bc}^0)$ is defined as
$P(\Delta, L_{bc}^0)x=\sum_{n\in\Omega}\mathbb{P}_nx$, $x\in L_2[0, \omega]$.
We consider also the set $\widetilde{\Delta}=\widetilde{\Delta}(\Omega)=\cup_{n\in\Omega}\sigma_n$, where $\sigma_n$ is defined
in Theorem~\ref{thasympt0and1H}, and denote by $\widetilde{\mathbb{P}}_n$ the Riesz projection constructed for
$\sigma_n$, $n\geq m+1$. Then the projection $P(\widetilde{\Delta}, L_{bc})$ is defined as
$P(\widetilde{\Delta}, L_{bc})x=\sum_{n\in\Omega}\widetilde{\mathbb{P}}_nx$, $x\in L_2[0, \omega]$.

Similarly, define the Riesz projections $P(\Delta, L_{dir}^0)$, $P(\widetilde{\Delta}, L_{dir})$:
$P(\Delta, L_{dir}^0)x=\sum_{n\in\Omega}P_{n, dir}x$, \linebreak
$P(\widetilde{\Delta}, L_{dir})x=\sum_{n\in \Omega}\widetilde{P}_{n, dir}x$, $x\in L_2[0, \omega]$.
In this case, $\Omega\subset\mathbb{N}\setminus\{1, \dots, m\}$ and $\widetilde{P}_{n, dir}$ is the Riesz projection constructed
for the singleton $\sigma_n=\{\widetilde{\lambda}_{n, dir}\}$, $n\geq m+1$. .

By $\|\cdot\|_2$ we denote the norm in the ideal $\mathfrak{S}_2(L_2[0, \omega])$ of Hilbert-Schmidt operators
(see~\cite[Ch.~3, Sec.~9]{<Gohberg>}).
\begin{theorem}\label{thPPH}
Let $m$ be as in Theorem~\ref{thasympt0and1H} or in Theorem~\ref{thasympt(0,1)H}.
For every subset $\Omega\subset \{m+1, m+2, \dots\}$ and $d(\Omega)=\min\limits_{n\in \Omega}n$ the following
estimates hold:
\[
\|P(\widetilde{\Delta}, L_{bc})-P(\Delta, L_{bc}^0)\|_2\leq C/d^{2k-3/2}(\Omega).
\]
\end{theorem}

Theorem~\ref{thPPH} implies directly

\begin{corollary}\label{cor3H}
The following estimates of the spectral decompositions for the operators $L_{bc}$ and $L_{bc}^0$ hold:
\[
\bigg\|P(\sigma_{(m)}, L_{bc})+\sum_{j=m+1}^n\widetilde{\mathbb{P}}_j -
\sum_{j=0}^n\mathbb{P}_j\bigg\|_2\leq C/n^{2k-3/2}, \; n\geq m+1, \; bc\in\{per, ap\},
\]
and
\[
\bigg\|P(\sigma_{(m)}, L_{dir}) + \sum_{j=m+1}^n \widetilde{P}_{j, dir} - \sum_{j=1}^nP_{j, dir}\bigg\|_2\leq
C/n^{2k-3/2}, \quad n\geq m+1,
\]
where $\sigma_{(m)}$ is the set from representation (\ref{sigmaLthetaH}).
\end{corollary}

Finally, we state the result on asymptotic behaviour of an analytic semigroup of operators related to the operator $-L_{bc}$.
\begin{theorem}\label{thsemigroup1H}
Let $m$ be as in Theorem~\ref{thasympt0and1H} or in Theorem~\ref{thasympt(0,1)H}.
The operator $-L_{bc}$, $bc\in\{per, ap, dir\}$, is sectorial and it generates an analytic semigroup of operators
$T: \mathbb{R}_+\to \mathrm{End}\,L_2[0, \omega]$. This semigroup is similar to a semigroup $\widetilde{T}:
\mathbb{R}_+\to\mathrm{End}\,L_2[0, \omega]$ of the following form
\[
\widetilde{T}(t)=\widetilde{T}_{(m)}(t)\oplus \widetilde{T}^{(m)}(t), \quad t\in\mathbb{R}_+,
\]
acting in $L_2[0, \omega]=\mathcal{H}_{(m)}\oplus\mathcal{H}^{(m)}$. Here $\mathcal{H}_{(m)}=\mathrm{Im}\,\mathbb{P}_{(m)}$,
$\mathcal{H}^{(m)}=\mathrm{Im}\,(I-\mathbb{P}_{(m)})$ for $bc\in\{per, ap\}$ and $\mathcal{H}_{(m)}=\mathrm{Im}\,P_{(m)}$,
$\mathcal{H}^{(m)}=\mathrm{Im}\,(I-P_{(m)})$ for $bc=dir$.

For $bc=dir$ the semigroup $\widetilde{T}^{(m)}:\mathbb{R}_+\to\mathrm{End}\,\mathcal{H}^{(m)}$ is represented as
\[
\widetilde{T}^{(m)}(t)x=\sum_{s\geq m+1}e^{-\widetilde{\lambda}_{s, dir}t}P_{s, dir}x, \quad x\in L_2[0, \omega],
\]
where the eigenvalues $\widetilde{\lambda}_{s, dir}$, $s\geq m+1$, are defined in (\ref{lambdantheta01H}).
\end{theorem}

In Section~\ref{sec4} it will be also given a representation of the semigroup
$\widetilde{T}^{(m)}: \mathbb{R}_+\to\mathrm{End}\,\mathcal{H}^{(m)}$ for $bc\in\{per, ap\}$.

Note that some of the results stated above have been partially announced in~\cite{PolyakovDE}.

\section{Spectral analysis of an abstract operators in Hilbert space}\label{sec2}

In this Section we apply the method of similar operators to abstract linear operators acting in a complex Hilbert space
$\mathcal{H}$ and having the same spectral properties as $L_{bc}$, $bc\in\{per, ap, dir\}$.
Recall some basic notions and results of this method in the form relevant to our purposes
(see \cite[Section~2]{<2011>} and \cite[Section~2]{Bask_Pol_2017}).

Let $\mathrm{End}\,\mathcal{H}$ be the Banach algebra of bounded linear operators acting in $\mathcal{H}$.
\begin{definition}
Two linear operators $A_i: D(A_i)\subset\mathcal{H}\to\mathcal{H}$, $i=1, 2$, are called \emph{similar} if
there exists a continuously invertible operator $U\in\mathrm{End}\,\mathcal{H}$ such that $UD(A_2)=D(A_1)$ and $A_1Ux=UA_2x$,
$x\in D(A_2)$. Such an operator $U$ is called a \emph{similarity transformation} from $A_1$ to $A_2$.
\end{definition}

Let $A: D(A)\subset\mathcal{H}\to\mathcal{H}$ be a linear closed operator and $\mathfrak{L}_A(\mathcal{H})$ denote the Banach
space of those operators on $\mathcal{H}$ that are subordinate to the operator $A$. By definition, a linear operator
$B: D(B)\subset\mathcal{H}\to\mathcal{H}$ belongs to $\mathfrak{L}_A(\mathcal{H})$ if $D(A)\subseteq D(B)$ and
$\|B\|_A=\inf\{C>0: \|Bx\|\leq C(\|x\|+\|Ax\|), x\in D(A)\}$ is finite. Note that $\|\cdot\|_A$ is the norm in
$\mathfrak{L}_A(\mathcal{H})$.

\begin{definition}[\cite{<2011>}]\label{def2}
Let $\mathfrak{U}$ be a linear subspace of $\mathfrak{L}_A(\mathcal{H})$ and
$J: \mathfrak{U}\to\mathfrak{U}$ and $\Gamma: \mathfrak{U}\to \mathrm{End}\,\mathcal{H}$ linear transformers
(i.~e., linear operators in the space of linear operators). The triple $(\mathfrak{U}, J, \Gamma)$ is
called \emph{admissible} for the operator $A$ and $\mathfrak{U}$ is called the \emph{space of admissible perturbations}
if the following conditions hold:

1) $\mathfrak{U}$ is a Banach space (with respect to the norm $\|\cdot\|_*$) continuously embedded into
$\mathfrak{L}_A(\mathcal{H})$;

2) $J$ and $\Gamma$ are continuous transformers and $J$ is a projection;

3) $(\Gamma X)D(A)\subset D(A)$ and $A\Gamma Xx-\Gamma XAx = (X-JX)x$ for all $X\in\mathfrak{U}$ and $x\in D(A)$;

4) $X\Gamma Y$, $\Gamma XY\in\mathfrak{U}$ for every $X, Y\in\mathfrak{U}$ and there exists a constant $\gamma>0$ such that
\[
\|\Gamma\|\leq\gamma, \quad \max\{\|X\Gamma Y\|_\ast, \|\Gamma XY\|_\ast\}\leq \gamma\|X\|_\ast\|Y\|_\ast;
\]

5) for any $X\in\mathfrak{U}$ and $\varepsilon>0$ there exists $\lambda_\varepsilon\in\rho(A)$ such that
$\|X(A-\lambda_\varepsilon I)^{-1}\|<\varepsilon$, where $\rho(A)$ is the resolvent set of the operator $A$.
\end{definition}
\begin{theorem}[\cite{<2011>}]\label{thsimilar}
Let $(\mathfrak{U}, J, \Gamma)$ be an admissible triple for the operator $A: D(A)\subset\mathcal{H}\to\mathcal{H}$ and
operator $B$ belong to $\mathfrak{U}$. If
\[
\|J\|\|B\|_\ast\|\Gamma\|<1/4,
\]
then the operator $A-B$ is similar to $A-JX_*$, where $X_*\in\mathfrak{U}$ is a solution of the nonlinear equation
\[
X=B\Gamma X-(\Gamma X)(JB)-(\Gamma X)J(B\Gamma X)+B=:\Phi(X).
\]
This solution can be found by the method of simple iterations by setting $X_0=0$, $X_1=B$ and $X_{n+1}=\Phi(X_n)$,
$n\geq 2$. Moreover, the operator $\Phi: \mathfrak{U}\to\mathfrak{U}$ is a contraction in the ball $\{X\in\mathfrak{U}:
\|X-B\|_*\leq 3\|B\|_*\}$ and the invertible operator
$I+\Gamma X_*\in\mathrm{End}\,\mathcal{H}$ is a similarity transformation from $A-B$ to $A-JX_*$.
\end{theorem}

Now we apply the method of similar operators to an abstract linear operator of the form $A_{bc}-B$, where
$A_{bc}: D(A_{bc})\subset\mathcal{H}\to\mathcal{H}$ is a normal operator (see~\cite[Ch.~1, Sec.~6]{<Kato>}) with
discrete spectrum. Our conditions on $B$ will be described below. The index $bc$ here stands to denote our assumption that
the spectrum $\sigma(A_{bc})$ has the following form $\lambda_n=\pi^{2k}(2n+\theta)^{2k}/\omega^{2k}$, $n\in\mathbb{Z}_+$,
$k>1$, for $bc\in\{per, ap\}$ and $\lambda_{n, dir}=\pi^{2k}n^{2k}/\omega^{2k}$, $n\in\mathbb{N}$, $k>1$, for $bc=dir$.
In fact, we assume that $A_{bc}$ and $L_{bc}^0$ have the same spectral properties.
Keeping this in mind, we will use the same notation for the eigenvalues and eigenfunctions of these two operators as well
as for the corresponding projections. In addition, for $bc\in\{per, ap\}$ the operators $A_{bc}$ will be denoted by $A_\theta$,
where $\theta=0$ and $\theta=1$ stands for $bc=per$ and $bc=ap$, respectively.

Let $\mathbb{J}$ denote one of the sets $\mathbb{Z}_+$ or $\mathbb{N}$. We suppose that the operator $B$ belongs to
the ideal of Hilbert-Schmidt operators $\mathfrak{S}_2(\mathcal{H})$
(see~\cite[Ch.~3, Sec.~9]{<Gohberg>}). By $\|X\|_2=\big(\sum_{n, j\in\mathbb{J}}|(Xg_n, g_j)|^2\big)^{1/2}$,
$X\in\mathfrak{S}_2(\mathcal{H})$, it is denoted the norm in $\mathfrak{S}_2(\mathcal{H})$. Here $g_n$, $n\in\mathbb{J}$,
are the eigenfunctions of the operator $A_{bc}$.

Let $\mathbb{P}_n$, $n\in\mathbb{Z}_+$, and $P_{n, dir}$, $n\in\mathbb{N}$, be the orthogonal Riesz projections constructed
for the sets $\{\lambda_n\}$ and $\{\lambda_{n, dir}\}$, respectively.
Therefore, $A_\theta\mathbb{P}_n=\lambda_n\mathbb{P}_n$ and $A_{dir}P_{n, dir}=\lambda_{n, dir}P_{n, dir}$.
For every $x\in\mathcal{H}$ these projections are defined by formulas (\ref{pndef}) and (\ref{pndirdef}).

We take the ideal of Hilbert-Schmidt operators $\mathfrak{S}_2(\mathcal{H})$ in the role of the space of admissible perturbations
$\mathfrak{U}$ and construct the transformers $J_{bc}, \Gamma_{bc}\in\mathrm{End}\,\mathfrak{S}_2(\mathcal{H})$
in the following way. If $bc\in\{per, ap\}$, then they are defined by
\[
J_{bc}X=\sum_{n=0}^\infty\mathbb{P}_nX\mathbb{P}_n \ \ \text{ and } \ \ \Gamma_{bc}X=\sum\limits_{\substack{s, j=0\\
\lambda_s\ne \lambda_j}}^\infty\frac{\mathbb{P}_sX\mathbb{P}_j}{\lambda_s - \lambda_j}, \quad X\in\mathfrak{S}_2(\mathcal{H}).
\]
If $bc=dir$, then they are determined by
\[
J_{dir}X=\sum_{n=1}^\infty P_{n, dir}XP_{n, dir} \ \ \text{ and } \ \ \Gamma_{dir}X=\sum\limits_{\substack{s, j=1 \\
s\ne j}}^\infty\frac{P_{s, dir}XP_{j, dir}}{\lambda_{s, dir} - \lambda_{j, dir}}, \quad X\in\mathfrak{S}_2(\mathcal{H}).
\]
Taking into account the definition of  $\mathfrak{S}_2(\mathcal{H})$ and using that
$\min_{\lambda_i\ne \lambda_j}|\lambda_i-\lambda_j|$ and $\min_{i\ne j}|\lambda_{i, dir}-\lambda_{j, dir}|$ are positive,
we conclude easily that the transformer $\Gamma_{bc}$, $bc\in\{per, ap, dir\}$, is well-defined and bounded on
$\mathfrak{S}_2(\mathcal{H})$.

Also, for $bc\in\{per, ap\}$ we consider the sequences of transformers $J_{m, bc}$ and $\Gamma_{m, bc}$, $m\in\mathbb{Z}_+$,
defined by
\begin{equation}\label{jmbc}
J_{m, bc}X=J_{bc}(X-\mathbb{P}_{(m)}X\mathbb{P}_{(m)}) + \mathbb{P}_{(m)}X\mathbb{P}_{(m)}, \quad
X\in\mathfrak{S}_2(\mathcal{H}),
\end{equation}
\begin{equation}\label{gmbc}
\Gamma_{m, bc}X=\Gamma_{bc}(X-\mathbb{P}_{(m)}X\mathbb{P}_{(m)}), \quad X\in\mathfrak{S}_2(\mathcal{H}),
\end{equation}
where $\mathbb{P}_{(m)}=\sum_{j=0}^m\mathbb{P}_j$. Note that $J_{bc}=J_{m, bc}$ and
$\Gamma_{bc}=\Gamma_{m, bc}$ for $m=0$.

For $bc=dir$ and $m\in\mathbb{N}$, define these transformers by
\begin{equation}\label{jmdir}
J_{m, dir}X=J_{dir}(X-P_{(m)}XP_{(m)}) + P_{(m)}XP_{(m)}, \quad X\in\mathfrak{S}_2(\mathcal{H}),
\end{equation}
\begin{equation}\label{gmdir}
\Gamma_{m, dir}X=\Gamma_{dir}(X-P_{(m)}XP_{(m)}), \quad X\in\mathfrak{S}_2(\mathcal{H}),
\end{equation}
where $P_{(m)}=\sum_{j=1}^mP_{j, dir}$.
\begin{lemma}\label{lhcorrect-1}
The transformers $J_{m, bc}$ and $\Gamma_{m, bc}$, $m\in\mathbb{J}$, are self-adjoint operators on
$\mathfrak{S}_2(\mathcal{H})$. Each transformer $J_{m, bc}$ is an orthogonal projection.
Moreover, $\|J_m\|_2=1$ and the following estimates hold:
\begin{equation}\label{5-gamma}
\|\Gamma_{m, bc}\|_2\leq \frac{\omega^{2k}}{\pi^{2k}(2m+1)}\cdot
\begin{cases}
1/4(2m+\theta)^{2k-2}, \, m\in\mathbb{Z}_+, \, bc\in\{per, ap\}, \\
1/m^{2k-2}, \quad m\in\mathbb{N}, \quad bc=dir.
\end{cases}
\end{equation}
\end{lemma}

The proof is similar to the proof of \cite[Lemma~2]{Bask_Pol_2017}.
\begin{lemma}\label{lhtriple}
For any $m\in\mathbb{J}$, $(\mathfrak{S}_2(\mathcal{H}), J_{m, bc}, \Gamma_{m, bc})$ is an admissible triple
for the operator $A_{bc}$, $bc\in\{per, ap, dir\}$. The constant $\gamma=\gamma(m)$ from
Definition~\ref{def2} has the following estimate:
\[
\|\Gamma_{m, bc}\|_2\leq \gamma(m)=\frac{\omega^{2k}}{\pi^{2k}(2m+1)}\cdot
\begin{cases}
1/4(2m+\theta)^{2k-2}, \quad &bc\in\{per, ap\}, \\
1/m^{2k-2}, \quad & bc=dir.
\end{cases}
\]
\end{lemma}

The proof is similar to the proof of \cite[Lemma~4]{Bask_Pol_2017}.

Now we are ready to state the main theorem on similarity.
\begin{theorem}\label{thsimilar-2}
Let $B$ belong to $\mathfrak{S}_2(\mathcal{H})$ and $m\in\mathbb{J}$ satisfy one of the following conditions:
\begin{equation}\label{9-estimate}
\frac{\omega^{2k}\|B\|_2}{\pi^{2k}(2m+1)(2m+\theta)^{2k-2}}<1 \quad \text{for} \quad \theta\in\{0, 1\}
\end{equation}
or
\begin{equation}\label{10-estimate}
\frac{\omega^{2k}\|B\|_2}{\pi^{2k}(2m+1)m^{2k-2}}<\frac{1}{4} \quad \text{for}\quad bc=dir.
\end{equation}
Then $A_{bc}-B$ is similar to $A_{bc}-J_{m, bc}X_*$, where
$X_*\in\mathfrak{S}_2(\mathcal{H})$ is a solution of the nonlinear equation
\begin{equation}\label{eq52}
X = B\Gamma_{m, bc}X - (\Gamma_{m, bc}X)(J_{m, bc}B) - (\Gamma_{m, bc}X)J_{m, bc}(B\Gamma_{m, bc}X)+B=:\Phi(X).
\end{equation}
This solution can be found by the method of simple iterations by setting $X_0=0$, $X_1=B, \dots$. Moreover, the operator $\Phi:
\mathfrak{S}_2(\mathcal{H})\to\mathfrak{S}_2(\mathcal{H})$ is a contraction in the ball
$\{X\in\mathfrak{S}_2(\mathcal{H}): \|X-B\|_2\leq 3\|B\|_2\}$, the operator $I+\Gamma_{m, bc}X_*$ is a
similarity transformation from $A_{bc}-B$ to $A_{bc}-J_{m, bc}X_*$, and
\[
A_{bc}-B=(I+\Gamma_{m, bc}X_*)(A_{bc}-J_{m, bc}X_*)(I+\Gamma_{m, bc}X_*)^{-1}.
\]
\end{theorem}
The proof follows from Lemma~\ref{lhtriple} and Theorem~\ref{thsimilar}.

Throughout the rest of this section, we will use the assumptions and notation of Theorem~\ref{thsimilar-2}.
\begin{theorem}\label{th2}
The operator $A_{bc}-B$ has a discrete spectrum and
\[
\begin{array}{c}
A_{bc}-J_{m, bc}X_*=A_{bc}-\mathbb{P}_{(m)}X_*\mathbb{P}_{(m)} -
\sum\limits_{n\geq m+1}\mathbb{P}_nX_*\mathbb{P}_n, \quad bc\in\{per, ap\}, \\
A_{dir}-J_{m, dir}X_*=A_{dir}-P_{(m)}XP_{(m)}-\sum\limits_{n\geq m+1}P_{n, dir}X_*P_{n, dir}.
\end{array}
\]
Furthermore,
\begin{equation}\label{14-new}
\sigma(A_{bc}-B)=\sigma(A_{(m)})\cup\Big(\bigcup_{n\geq m+1}\sigma(A_n)\Big)=\sigma_{(m)}\cup
\Big(\bigcup_{n\geq m+1}\sigma_n\Big),
\end{equation}
where the sets $\sigma_{(m)}$ and $\sigma_n$, $n\geq m+1$, are mutually disjoint and $A_{(m)}$ and $A_n$ are defined as follows.
For $bc\in\{per, ap\}$, $A_{(m)}$ is the restriction of the operator $A_{bc}-\mathbb{P}_{(m)}X_*\mathbb{P}_{(m)}$ to the
invariant subspaces $\mathcal{H}_{(m)}=\mathrm{Im}\,\mathbb{P}_{(m)}$ and $A_n$ is the restriction of the operator
$A_{bc}-\mathbb{P}_nX_*\mathbb{P}_n$ to the subspace $\mathcal{H}_n=\mathrm{Im}\,\mathbb{P}_n$. If $bc=dir$, then $A_{(m)}$ is
the restriction of the operator $A_{bc}-P_{(m)}X_*P_{(m)}$ on the invariant subspaces $\mathcal{H}_{(m)}=\mathrm{Im}\,P_{(m)}$
and $A_n$ is the restriction of the operator $A_{bc}-P_{n, dir}X_*P_{n, dir}$ on the subspace $\mathcal{H}_n=\mathrm{Im}\,P_{n, dir}$.
\end{theorem}
\begin{proof}
Obviously, the operator $A_{bc}-J_{m, bc}X_*$ has a discrete spectrum. Therefore, the operator $A_{bc}-B$ also has a
discrete spectrum. Hence, the spectrums of these operators coincide. Next, the subspaces $\mathcal{H}_{(m)}$ and $\mathcal{H}_n$
are invariant with respect to the operator $A_{bc}-J_{m, bc}X_*$. Hence, the right-hand part in (\ref{14-new}) is a subset of
$\sigma(A_{bc}-J_{m, bc}X_*)=\sigma(A_{bc}-B)$. The proof of the reverse embedding can be found in \cite[Sec.~4]{<2011>}.
\end{proof}

Now we can state the main theorems of Section~\ref{sec2} concerning the asymptotic behavior
of the eigenvalues of the operator $A_{bc}-B$, $bc\in\{per, ap, dir\}$.
\begin{theorem}\label{th4-asymptotic}
Let the spectrum of $A_{bc}-B$, $bc\in\{per, ap\}$, has the representation (\ref{14-new}) such that $\sigma_{(m)}$ is a
finite set with number of points not exceeding $2m+1$ and the sets $\sigma_n$, $n\geq m+1$, consist of at most two points.
Every set $\sigma_n$, $n\geq m+1$, coincides with the
spectrum of a matrix $\mathcal{A}_n$ having the following form
\begin{equation}\label{20-asympt}
\mathcal{A}_n=\frac{\pi^{2k}(2n+\theta)^{2k}}{\omega^{2k}}
\begin{pmatrix}
1 & 0 \\
0 & 1
\end{pmatrix}
-\mathcal{B}_n^\theta+\mathcal{C}_n,
\end{equation}
where
\[
\mathcal{B}_n^\theta=
\begin{pmatrix}
(Be_{-n-\theta}, e_{-n-\theta}) & (Be_n, e_{-n-\theta}) \\
(Be_{-n-\theta}, e_n) & (Be_n, e_n)
\end{pmatrix}
\]
and for all $n\geq n_0:=\max\{m+1, (3\|B\|_2\omega^{2k}\pi^{-2k}/2)^{1/(2k-1)}\}$, the Hilbert-Schmidt norm of the matrix
$\mathcal{C}_n$ satisfies the estimate
\begin{equation}\label{cnpn}
\|\mathcal{C}_n\|_2\leq
\frac{\omega^{2k}}{2\pi^{2k}n(2n+\theta)^{2k-2}}\|\mathbb{P}_nB-\mathbb{P}_nB\mathbb{P}_n\|_2
\|B\mathbb{P}_n-\mathbb{P}_nB\mathbb{P}_n\|_2.
\end{equation}
\end{theorem}
\begin{proof}
Applying the projection $\mathbb{P}_n$ to the left-hand and right-hand parts of the equation (\ref{eq52}) with $X=X_*$, we obtain
\begin{equation}\label{23-equation}
\mathbb{P}_nX_*\mathbb{P}_n=\mathbb{P}_nB\mathbb{P}_n+\mathbb{P}_n(B\Gamma_{m, bc}X_*)\mathbb{P}_n,
\quad n\geq m+1.
\end{equation}
Next, for all operators $X, Y\in\mathfrak{S}_2(\mathcal{H})$ we have the following equalities :
\begin{flalign*}
(J_{m, bc}&X)\mathbb{P}_n=\mathbb{P}_n(J_{bc}X)\mathbb{P}_n=\mathbb{P}_nX\mathbb{P}_n, \quad
\mathbb{P}_n(J_{m, bc}X)(\Gamma_{m, bc}Y)\mathbb{P}_n=0, \\
&\Gamma_{m, bc}(\mathbb{P}_nX\mathbb{P}_n)=0, \quad \mathbb{P}_n((J_{m, bc}X)\Gamma_{m, bc}Y)\mathbb{P}_n=0,
\quad n\geq m+1.
\end{flalign*}
Then
\[
\mathbb{P}_n(B\Gamma_{m, bc}X_*)\mathbb{P}_n = \mathbb{P}_n(B-J_{m, bc}B)(\Gamma_{m, bc}X_*)\mathbb{P}_n=
\mathbb{P}_n(B - \mathbb{P}_nB\mathbb{P}_n)(\Gamma_{m, bc}X_*)\mathbb{P}_n
\]
and, consequently,
\begin{equation}\label{pnbxpn}
\|\mathbb{P}_n(B\Gamma_{m, bc}X_*)\mathbb{P}_n\|_2\leq \|\mathbb{P}_nB-\mathbb{P}_nB\mathbb{P}_n\|_2
\|(\Gamma_{m, bc}X_*)\mathbb{P}_n\|_2, \quad n\geq m+1.
\end{equation}

Now we estimate the second factor in (\ref{pnbxpn}). Using the equalities
$(\Gamma_{m, bc}X_*)\mathbb{P}_n=\linebreak \Gamma_{m, bc}(X_*-J_{m, bc}X_*)\mathbb{P}_n$, $n\geq m+1$, we get
\[
\|(\Gamma_{m, bc}X_*)\mathbb{P}_n\|_2\leq \frac{\|(X_*-J_{m, bc}X_*)
\mathbb{P}_n\|_2}{\min\limits_{j\ne n}|\lambda_j-\lambda_n|}=
d_n\|X_*\mathbb{P}_n-\mathbb{P}_nX_*\mathbb{P}_n\|_2,
\]
where
\[
d_n=\frac{1}{\min\limits_{j\ne n}|\lambda_j-\lambda_n|}\leq \frac{\omega^{2k}}{4\pi^{2k}(2n-1)(2n+\theta)^{2k-2}}\leq
\frac{\omega^{2k}}{4\pi^{2k}n(2n+\theta)^{2k-2}}.
\]

Since the operator $X_*$ satisfies (\ref{eq52}) and $\mathbb{P}_nX_*\mathbb{P}_n=\mathbb{P}_nB\mathbb{P}_n+
\mathbb{P}_nB\Gamma_{m, bc}(X_*-J_{m, bc}X_*)\mathbb{P}_n$, $n\geq m+1$, we have
\begin{flalign*}
(X_* &- J_{m, bc}X_*)\mathbb{P}_n = X_*\mathbb{P}_n-\mathbb{P}_nX_*\mathbb{P}_n =
(B-\mathbb{P}_nB\mathbb{P}_n)\mathbb{P}_n \\
&+ B\Gamma_{m, bc}(X_*-\mathbb{P}_nX_*\mathbb{P}_n)\mathbb{P}_n -
\Gamma_{m, bc}(X_* - \mathbb{P}_nX_*\mathbb{P}_n)\mathbb{P}_nB\mathbb{P}_n-
\mathbb{P}_nB\Gamma_{m, bc}(X_*-\mathbb{P}_nX_*\mathbb{P}_n)\mathbb{P}_n.
\end{flalign*}
Therefore,
\begin{flalign*}
\|X_*\mathbb{P}_n &- \mathbb{P}_nX_*\mathbb{P}_n\|_2\leq
\|B\mathbb{P}_n-\mathbb{P}_nB\mathbb{P}_n\|_2 + d_n\|B\|_2\|X_*\mathbb{P}_n-\mathbb{P}_nX_*\mathbb{P}_n\|_2 +
d_n\|B\|_2\|X_*\mathbb{P}_n-\mathbb{P}_nX_*\mathbb{P}_n\| \\
&+ d_n\|B\|_2\|X_*\mathbb{P}_n-\mathbb{P}_nX_*\mathbb{P}_n\|_2 = \|B\mathbb{P}_n-\mathbb{P}_nB\mathbb{P}_n\|_2 +
3d_n\|B\|_2\|X_*\mathbb{P}_n-\mathbb{P}_nX_*\mathbb{P}_n\|_2.
\end{flalign*}
Thus, for $n\in\mathbb{N}$ such that $3d_n\|B\|_2\leq 1/2$,
\[
\|X_*\mathbb{P}_n-\mathbb{P}_nX_*\mathbb{P}_n\|_2\leq 2\|B\mathbb{P}_n-\mathbb{P}_nB\mathbb{P}_n\|_2.
\]

Consequently,
\begin{equation}\label{gx}
\|(\Gamma_{m, bc}X_*)\mathbb{P}_n\|_2 = \|\Gamma_{m, bc}(X_*\mathbb{P}_n-\mathbb{P}_nX_*\mathbb{P}_n)\|_2\leq
2d_n\|B\mathbb{P}_n - \mathbb{P}_nB\mathbb{P}_n\|_2,
\end{equation}
which is the desired estimation of the second factor in (\ref{pnbxpn}).

Using (\ref{pnbxpn}) and (\ref{gx}), we get
\begin{equation}\label{cnnew}
\|\mathbb{P}_n(X_*-B)\mathbb{P}_n\|_2 \leq \frac{\omega^{2k}}{2\pi^{2k}(2n-1)(2n+\theta)^{2k-2}}
\|\mathbb{P}_nB-\mathbb{P}_nB\mathbb{P}_n\|_2\|B\mathbb{P}_n-\mathbb{P}_nB\mathbb{P}_n\|_2,
\quad n\geq n_0.
\end{equation}

Since the restrictions of the operators from (\ref{23-equation}) to $\mathcal{H}_n$ coincide with matrices $\mathcal{B}_n$ and
$\mathcal{C}_n$, the matrix $\mathcal{C}_n$ satisfies the estimation (\ref{cnnew}). This proves Theorem~\ref{th4-asymptotic}.
\end{proof}
\begin{theorem}\label{th3}
Let the spectrum of the operator $A_{dir}-B$ has the form (\ref{14-new}) such that $\sigma_{(m)}$ is a finite set with
number of points not exceeding $m$ and $\sigma_n$ is a singleton $\{\widetilde{\lambda}_{n, dir}\}$, $n\geq m+1$. The eigenvalues
$\widetilde{\lambda}_{n, dir}$, $n\geq n_1$, have the following asymptotic representation
\[
\widetilde{\lambda}_{n, dir}=\pi^{2k}n^{2k}\omega^{-2k} -(Be_{n, dir}, e_{n, dir}) +
\eta_n, \quad n\geq n_1,
\]
where the sequence $(\eta_n)$ satisfies the estimate
\begin{equation}\label{etapnbpn}
|\eta_n|\leq \frac{2\omega^{2k}}{\pi^{2k}n^{2k-1}}
\|P_{n, dir}B-P_{n, dir}BP_{n, dir}\|_2\|BP_{n, dir}-P_{n, dir}BP_{n, dir}\|_2, \quad n\geq n_1,
\end{equation}
and $n_1=\max\{m+1, (6\|B\|_2\omega^{2k}\pi^{-2k})^{1/(2k-1)}\}$.
\end{theorem}

The proof is similar to the proof of Theorem~\ref{th4-asymptotic}.

\section{Preliminary similarity transformation of the operator $L_{bc}$}\label{sec3}

In this section we apply the method of similar operators to the operator $L_{bc}=L_{bc}^0-Q$, $bc\in\{per, ap, dir\}$ in the
Hilbert space $\mathcal H:=L_2[0, \omega]$. In the previous section we already construct the
admissible triple $(\mathfrak{S}_2(\mathcal{H}), J_m, \Gamma_m)$, but in this case we cannot apply directly the results of
Section~\ref{sec2} since the perturbation $Q$ does not belong to
$\mathfrak{S}_2(\mathcal{H})$. In this situation, following \cite[Sec.~2]{<2011>}, we will use a preliminary similarity
transformation to transform the operator $L_{bc}=L_{bc}^0-Q$ into an operator of the form $L_{bc}^0-\widetilde{Q}$,
where $\widetilde{Q}\in\mathfrak{S}_2(\mathcal{H})$.

Since the operator $Q$ belongs to $\mathfrak{L}_{L_{bc}^0}(\mathcal{H})$, the transformers
$J_{m, bc}$, $\Gamma_{m, bc}$ are well-defined by the formulas (\ref{jmbc})~--~(\ref{gmdir}). Here we take into account that
$P_sQP_j$, $s, j\in\mathbb{Z}$, are Hilbert-Schmidt operators.
\begin{remark}
Since a shift of the potential $q$ to a constant shifts the spectrum to the same constant, in proving our results we may
assume without loss of generality that $q_0=0$.
\end{remark}

Now we define the matrix of operators $Q$. If $bc\in\{per, ap\}$, then the matrix of operators
has the form $(b_{sj})$, $s, j\in\mathbb{Z}$, where
\begin{equation}\label{bsj}
b_{sj} = \frac{1}{\omega}\int_0^\omega q(t)e_s(t)\overline{e_j(t)}\,dt=
\frac{1}{\omega}\sum_{l\in\mathbb{Z}}q_l\int_0^\omega e^{i2\pi lt/\omega}e^{-i\pi(2s+\theta)t/\omega}
e^{i\pi(2j+\theta)t/\omega}\,dt=q_{s-j}.
\end{equation}

In the case $bc=dir$ we compute the matrix coefficients $q_{sj}$, $s, j\in\mathbb{N}$, of the operator $Q$. We get
\begin{flalign*}
q_{sj} &= \frac{1}{\omega}\int_0^\omega q(t)e_{s, dir}(t)\overline{e_{j, dir}(t)}\,dt = \frac{2\sqrt{2}}{\omega}
\sum_{l=1}^\infty \widetilde{q}_l\int_0^\omega
\cos\frac{\pi l}{\omega}t\sin\frac{\pi s}{\omega}t\sin\frac{\pi j}{\omega}\,dt \\
&= \frac{\sqrt{2}}{\omega}\sum_{l=1}^\infty \widetilde{q}_l\int_0^\omega
\cos\frac{\pi l}{\omega}t\big(\cos\frac{\pi}{\omega}(s-j)t - \cos\frac{\pi}{\omega}(s+j)t\big)\,dt \\
&= \frac{\sqrt{2}}{\omega}(\widetilde{q}_{|s-j|}-\widetilde{q}_{s+j})\int_0^\omega\cos^2\frac{\pi l}{\omega}t\,dt =
\frac{1}{\sqrt{2}}(\widetilde{q}_{|s-j|}-\widetilde{q}_{s+j}).
\end{flalign*}
Thus,
\begin{equation}\label{qij}
q_{sj}=\frac{1}{\sqrt{2}}(\widetilde{q}_{|s-j|}-\widetilde{q}_{s+j}), \quad s, j\in\mathbb{N}.
\end{equation}

The following lemmas are needed in the sequel.
\begin{lemma}\label{lhgv-gsH}
The operators $\Gamma_{m, bc}Q$, $m\in\mathbb{J}$, $bc\in\{per, ap, dir\}$, are Hilbert-Schmidt ones and
there exists $m\in\mathbb{J}$ such that $\|\Gamma_{m, bc}Q\|_2<1$. Moreover, the following estimates hold:
\begin{flalign}
&\|\mathbb{P}_n(\Gamma_{m, bc}Q)\|_2=\|(\Gamma_{m, bc}Q)\mathbb{P}_n\|_2\leq
\frac{\omega^{2k}}{\pi^{2k}\sqrt{2}}\frac{\|q\|_2}{(2n+\theta)^{2k-1}}, \quad m, n\in\mathbb{Z}_+, \label{pngvH}\\
&\|P_{n, dir}(\Gamma_{m, dir}Q)\|_2=\|(\Gamma_{m, dir}Q)P_{n, dir}\|_2\leq
\frac{\omega^{2k}\|q\|_2}{\pi^{2k}n^{2k-1}}, \quad m, n\in\mathbb{N}. \label{pndirgqH}
\end{flalign}
\end{lemma}
\begin{proof}
Let $bc\in\{per, ap\}$. Using (\ref{bsj}), we have
\begin{flalign}\label{gbhs}
\sum_{s, j\in\mathbb{Z}}|(\Gamma_{bc}Qe_j, e_s)|^2 &= \sum_{s, j\in\mathbb{Z}}\frac{|q_{s-j}|^2}{|\lambda_s-\lambda_j|^2}=
\frac{\omega^{4k}}{\pi^{4k}}\sum_{s, j\in\mathbb{Z}}
\frac{|q_{s-j}|^2}{((2s+\theta)^k-(2j+\theta)^k)^2((2s+\theta)^k+(2j+\theta)^k)^2} \nonumber\\
&\leq \frac{\omega^{4k}}{4\pi^{4k}}\sum_{s\in\mathbb{Z}}\frac{1}{(2s+\theta)^{2k}}\sum_{j\in\mathbb{Z}}
\frac{|q_{s-j}|^2}{(s-j)^2((2s+\theta)^{k-1}+\dots+(2j+\theta)^{k-1})^2} \nonumber\\
&\leq \frac{\omega^{4k}\|q\|_2^2}{4\pi^{4k}}\sum_{s\in\mathbb{Z}}\frac{1}{(2s+\theta)^{4k-2}}<\infty,
\end{flalign}
where $\|q\|_2^2=\sum_{s\in\mathbb{Z}}|q_s|^2$. Therefore, $\Gamma_{bc}Q\in\mathfrak{S}_2(\mathcal{H})$.

Next, since the operators $\Gamma_{m, bc}Q$ differ from $\Gamma_{bc}Q$ by some finite rank operator (see formula (\ref{gmbc})),
then $\Gamma_{m, bc}Q\in\mathfrak{S}_2(\mathcal{H})$. Taking $s=\pm n$ and $j=\pm n$ in (\ref{gbhs}), we obtain
(\ref{pngvH}) and (\ref{pndirgqH}), respectively.

Moreover, from (\ref{gmbc}) we get
\[
\lim_{m\to\infty}\|\Gamma_{m, bc}Q\|_2^2 = \lim_{m\to\infty}\|\Gamma_{bc}Q-\mathbb{P}_{(m)}(\Gamma_{bc}Q)\mathbb{P}_{(m)}\|_2^2=
\lim_{m\to\infty}\sum_{s, j\geq m+1}\|\mathbb{P}_s(\Gamma_{m, bc}Q)\mathbb{P}_j\|_2^2=0.
\]
Hence, we can find an $m\in\mathbb{Z}_+$ such that $\|\Gamma_{m, bc}Q\|_2<1$.

Similar arguments with (\ref{qij}) and (\ref{gmbc}) prove the result in the case $bc=dir$.
\end{proof}

\begin{lemma}\label{lhvgv-gsH}
The operators $Q\Gamma_{m, bc}Q$ and $(\Gamma_{m, bc}Q)J_{m, bc}Q$, $m\in\mathbb{J}$, are Hilbert-Schmidt ones
and the following estimates hold:
\begin{flalign}
&\|(Q\Gamma_{m, bc}Q)\mathbb{P}_n\|_2\leq \frac{\omega^{2k}\|q\|_2}{2\sqrt{3}\pi^{2k-1}(2n+\theta)^{2k-2}}\alpha(2n+\theta),
\label{vgvpnH} \\
&\|\mathbb{P}_n(Q\Gamma_{m, bc}Q)\|_2\leq C, \label{pnqgq}\\
&\|(Q\Gamma_{m, dir}Q)P_{n, dir}\|_2\leq \frac{2\omega^{2k}\|q\|_2}{\pi^{2k-1}n^{2k-2}\sqrt{3}}\beta(2n),\label{qgdirqpn} \\
&\|P_{n, dir}(Q\Gamma_{m, dir}Q)\|_2\leq C \label{pnggdirq}.
\end{flalign}
\end{lemma}
\begin{proof}
First we consider the case $bc\in\{per, ap\}$ and, for the reader's convenience, divide the proof into several steps.

1) Prove that $Q\Gamma_{bc}Q$ belongs to $\mathfrak{S}_2(\mathcal{H})$.  Using~(\ref{bsj}) and \cite[Lemma~7]{PolyakovAA},
we have
\begin{flalign}\label{qgq}
&\sum_{s, p\in\mathbb{Z}}|(Q\Gamma_{bc}Qe_s, e_p)|^2=\sum_{s, p\in\mathbb{Z}}
\bigg|\sum_{j\in\mathbb{Z}}\frac{q_{s-j}q_{j-p}}{\lambda_j-\lambda_p}\bigg|^2 \nonumber\\
&= \frac{\omega^{4k}}{\pi^{4k}}\sum_{s, p\in\mathbb{Z}}\bigg|\sum_{j\in\mathbb{Z}}\frac{q_{s-j}q_{j-p}}{((2j+\theta)^2-(2p+\theta)^2)
((2j+\theta)^{2k-2}+\dots+(2p+\theta)^{2k-2})}\bigg|^2\nonumber\\
&\leq \frac{\omega^{4k}}{4\pi^{4k}}\sum_{p\in\mathbb{Z}}\frac{1}{(2p+\theta)^{4k-4}}
\sum_{s\in\mathbb{Z}}\bigg|\sum_{j\in\mathbb{Z}}\frac{q_{s-j}q_{j-p}}{(j-p)(j+p+\theta)}\bigg|^2\leq C<\infty.
\end{flalign}
Therefore, the operator $Q\Gamma_{bc}Q$ belongs to
$\mathfrak{S}_2(\mathcal{H})$. Since the operators $Q\Gamma_{m, bc}Q$ differ from $Q\Gamma_{bc}Q$ by some finite rank
operator, $Q\Gamma_{m, bc}Q\in\mathfrak{S}_2(\mathcal{H})$.

2) Now we obtain the estimate (\ref{vgvpnH}). Taking in (\ref{qgq}) $s=n$ and $s=-n-\theta$, we have
\begin{flalign}\label{mainqgq}
\|(Q&\Gamma_{m, bc}Q)\mathbb{P}_n\|_2^2 = \|(Q\Gamma_{m, bc}Q)P_{-n-\theta}+(Q\Gamma_{m, bc}Q)P_n\|_2^2 \nonumber\\
&\leq 2\|(Q\Gamma_{m, bc}Q)P_{-n-\theta}\|_2^2+2\|(Q\Gamma_{m, bc}Q)P_n\|_2^2\leq 4\|(Q\Gamma_{m, bc}Q)P_n\|_2^2.
\end{flalign}
Let us estimate the norm in the right-hand side of (\ref{mainqgq}). Using formula (\ref{bsj}) and H\"older's inequality, we get
\begin{flalign*}
\|&(Q\Gamma_{bc}Q)P_n\|_2^2 = \sum_{s\in\mathbb{Z}}\bigg|\sum_{j\in\mathbb{Z}}\frac{q_{s-j}q_{j-n}}{\lambda_j-\lambda_n}\bigg|^2\\
&= \frac{\omega^{4k}}{\pi^{4k}}\sum_{s\in\mathbb{Z}}\bigg|\sum_{j\in\mathbb{Z}}\frac{q_{s-j}q_{j-n}}{((2j+\theta)^2-(2n+\theta)^2)
((2j+\theta)^{2k-2}+\dots+(2n+\theta)^{2k-2})}\bigg|^2 \\
&\leq \frac{\omega^{4k}}{16\pi^{4k}(2n+\theta)^{4k-4}}\sum_{s\in\mathbb{Z}}\bigg|\sum_{j\in\mathbb{Z}}
\frac{q_{s-j}q_{j-n}}{(j-n)(j+n+\theta)}\bigg|^2 =
\frac{\omega^{4k}}{16\pi^{4k}(2n+\theta)^{4k-4}}\sum_{s\in\mathbb{Z}}
\bigg|\sum_{p\in\mathbb{Z}}\frac{q_{s-n-p}q_p}{p(p+2n+\theta)}\bigg|^2\\
&\leq \frac{\omega^{4k}}{16\pi^{4k}(2n+\theta)^{4k-4}}\sum_{s\in\mathbb{Z}}\bigg(\sum_{p\in\mathbb{Z}}\frac{|q_{s-n-p}|^2}{p^2}\bigg)
\bigg(\sum_{p\in\mathbb{Z}}\frac{|q_p|^2}{(p+2n+\theta)^2}\bigg) \\
&\leq \frac{\omega^{4k}\|q\|_2^2}{48\pi^{4k-2}(2n+\theta)^{4k-4}}\sum_{j\in\mathbb{Z}}\frac{|q_{j-2n-\theta}|^2}{j^2}.
\end{flalign*}
Taking the square summable sequence
\begin{equation}\label{alpha}
\alpha(n)=\bigg(\frac{\|q\|_2^2}{n^2} + \sum\limits_{|p|\leq n, \; p\ne 0}\frac{|q_{p-n}|^2}{p^2}\bigg)^\frac{1}{2},
\quad n\in\mathbb{N},
\end{equation}
we then have
\begin{flalign*}
\|&(Q\Gamma_{bc}Q)P_n\|_2^2\leq \frac{\omega^{4k}\|q\|_2^2}{48\pi^{4k-2}(2n+\theta)^{4k-4}}
\bigg(\sum\limits_{\substack{|j|\leq 2n \\ j\ne 0}}\frac{|q_{j-2n-\theta}|^2}{j^2} +
\sum_{|j|\geq 2n+1}\frac{|q_{j-2n-\theta}|^2}{j^2}\bigg)\\
&\leq \frac{\omega^{4k}\|q\|_2^2}{48\pi^{4k-2}(2n+\theta)^{4k-4}}\bigg(\sum\limits_{\substack{|j|\leq 2n \\ j\ne 0}}
\frac{|q_{j-2n-\theta}|^2}{j^2} +
\frac{\|q\|_2^2}{(2n+\theta)^2}\bigg)=\frac{\omega^{4k}\|q\|_2^2}{48\pi^{4k-2}(2n+\theta)^{4k-4}}\alpha^2(2n+\theta).
\end{flalign*}

From (\ref{gmbc}) it follows that the same estimates hold for the operator $(Q\Gamma_{m, bc}Q)P_n$.
Combining this and (\ref{mainqgq}), we obtain (\ref{vgvpnH}).

3) Next let us estimate $\|P_n(Q\Gamma_{bc}Q)\|_2$, $n\in\mathbb{Z}$.
We have
\begin{flalign*}
\|P_n(Q\Gamma_{bc}Q)\|_2^2 &=\sum_{p\in\mathbb{Z}}\bigg|\sum_{j\in\mathbb{Z}}\frac{q_{n-j}q_{j-p}}{\lambda_j-\lambda_p}\bigg|^2
\leq \frac{\omega^{4k}}{16\pi^{4k}}\sum_{p\in\mathbb{Z}}\frac{1}{(2p+\theta)^{4k-4}}\bigg|\sum_{j\in\mathbb{Z}}
\frac{q_{n-j}q_{j-p}}{(j-p)(j+p+\theta)}\bigg|^2\\
&= \frac{\omega^{4k}}{16\pi^{4k}}\sum_{p\in\mathbb{Z}}
\frac{1}{(2p+\theta)^{4k-4}}\bigg|\sum_{s\in\mathbb{Z}}\frac{q_{n-s-p}q_s}{s(s+2p+\theta)}\bigg|^2\\
&\leq \frac{\omega^{4k}}{16\pi^{4k}}\sum_{p\in\mathbb{Z}}\frac{1}{(2p+\theta)^{4k-4}}
\bigg(\sum_{s\in\mathbb{Z}}\frac{|q_{n-s-p}|^2}{s^2}\bigg)\bigg(\sum_{s\in\mathbb{Z}}\frac{|q_s|^2}{(s+2p+\theta)^2}\bigg)\\
&\leq \frac{\omega^{4k}}{16\pi^{4k}}\sum_{p\in\mathbb{Z}}\frac{1}{(2p+\theta)^{4k-4}}
\bigg(\sum_{s\in\mathbb{Z}}\frac{|q_{n-s-p}|^2}{s^2}\bigg)\bigg(\sum_{j\in\mathbb{Z}}\frac{|q_{j-2p-\theta}|^2}{j^2}\bigg)\\
&\leq \frac{\omega^{4k}\|q\|_2^4}{144\pi^{4k-4}}\sum_{p\in\mathbb{Z}}\frac{1}{(2p+\theta)^{4k-4}}\leq C.
\end{flalign*}

The same estimates hold for the operator $P_n(Q\Gamma_{m, bc}Q)$. Using this and (\ref{mainqgq}) with
$P_n(Q\Gamma_{m, bc}Q)$ instead of $(Q\Gamma_{m, bc}Q)P_n$, we get (\ref{pnqgq}).

Let now $bc=dir$.
In this case the proof is similar to the previous one and we omit it. Note only that to check the estimates (\ref{qgdirqpn}) and
(\ref{pnggdirq}) one should take, instead of $\alpha(n)$, the square summable sequence
\begin{equation}\label{beta}
\beta(n)=\bigg(\frac{\|q\|_2^2}{n^2} + \sum\limits_{|p|\leq n, p\ne 0}
\frac{\widetilde{q}(p, n)}{p^2}\bigg)^\frac{1}{2}, \quad n\in\mathbb{N},
\end{equation}
where $\widetilde{q}(p, n)=\max\{|\widetilde{q}_{|p+n|}|^2, |\widetilde{q}_{|p-n|}|^2\}$, $p\in\mathbb{Z}, n\in\mathbb{N}$,
and $\widetilde{q}_s$, $s\in\mathbb{N}$, are Fourier coefficients of potential~$q$.

It remains to show that, for $m\in\mathbb{J}$ and $bc\in\{per, ap, dir\}$,  the operators $(\Gamma_{m, bc}Q)J_{m, bc}Q$
belong to $\mathfrak{S}_2(\mathcal{H})$. Evidently, $J_{m, bc}Q$, $m\in\mathbb{J}$, are bounded. From Lemma~\ref{lhgv-gsH}
it follows that the operators $\Gamma_{m, bc}Q$, $m\in\mathbb{J}$, belong to $\mathfrak{S}_2(\mathcal{H})$. Therefore,
$(\Gamma_{m, bc}Q)J_{m, bc}Q$, $m\in\mathbb{J}$, are Hilbert-Schmidt operators, which completes the proof.
\end{proof}

The following lemma contains all we need for the preliminary similarity transformation in our situation.
\begin{lemma}\label{lhpred1}
There exists $m\in\mathbb{J}$ such that the operators $Q$, $J_{m, bc}Q$, and $\Gamma_{m, bc}Q$ with $bc\in\{per, ap, dir\}$
satisfy the following conditions:\\
(a) $\Gamma_{m, bc}Q\in\mathrm{End}\,\mathcal{H}$ and $\|\Gamma_{m, bc}Q\|_2<1$;\\
(b) $(\Gamma_{m, bc}Q)D(L_{bc}^0)\subset D(L_{bc}^0)$;\\
(c) $Q\Gamma_{m, bc}Q$, $(\Gamma_{m, bc}Q)J_{m, bc}Q\in\mathfrak{S}_2(\mathcal{H})$;\\
(d) $L_{bc}^0(\Gamma_{m, bc}Q)x-(\Gamma_{m, bc}Q)L_{bc}^0x=(Q-J_{m, bc}Q)x$, $x\in D(L_{bc}^0)$;\\
(e) for every $\varepsilon>0$ there exists $\lambda_\varepsilon\in\rho(L_{bc}^0)$ with
$\|Q(L_{bc}^0-\lambda_\varepsilon I)^{-1}\|<\varepsilon$.
\end{lemma}
\begin{proof}
Properties $(a)$ and $(c)$ follow from~Lemmas~\ref{lhgv-gsH} and \ref{lhvgv-gsH}.
The proof of properties $(b)$, $(d)$, and $(e)$ is similar to \cite[Lemma~7]{Bask_Pol_2017}.
\end{proof}

Lemma~\ref{lhpred1} and \cite[Theorem~9]{Bask_Pol_2017} yield the first theorem on similarity.
\begin{theorem}\label{thsimilar-3H}
Let $l\in\mathbb{J}$ be such that $\|\Gamma_{l, bc}Q\|_2\leq 1/2$. Then the operator $L_{bc}^0-Q$ is similar to
the operator $L_{bc}^0-Q_0$, where $Q_0$ is defined by
\begin{equation}\label{v0thH}
Q_0=J_{l, bc}Q+(I+\Gamma_{l, bc}Q)^{-1}(Q\Gamma_{l, bc}Q-(\Gamma_{l, bc}Q)J_{l, bc}Q)=:J_{l, bc}Q + B.
\end{equation}
The operator $B$ is Hilbert-Schmidt one and the following representation holds
\[
(L_{bc}^0-Q)(I+\Gamma_{l, bc}Q)=(I+\Gamma_{l, bc}Q)(L_{bc}^0-Q_0).
\]
\end{theorem}
Note that the operator $J_{l, bc}Q$ does not belong to $\mathfrak{S}_2(\mathcal{H})$. In this relation take
$\widetilde{L}_{bc}^0=L_{bc}^0-J_{l, bc}Q$ as the unperturbed operator we need. Obviously, this operator is a normal one.
From Theorem~\ref{thsimilar-3H} it then follows that the operator $L_{bc}$ is similar to $\widetilde{L}_{bc}^0-B$, where
$B\in\mathfrak{S}_2(\mathcal{H})$. This provide the possibility to apply the scheme of method of similar operators from
Section~\ref{sec2} to this operator, which leads to the second theorem on similarity.

\begin{theorem}\label{thsimilar4H}
There exist numbers $l\in\mathbb{J}$ and $m\geq l+1$, such that the conditions (\ref{9-estimate}) or (\ref{10-estimate}) hold and, consequently, the operator $L_{bc}$, $bc\in\{per, ap, dir\}$, is similar to the operator $\widetilde{L}_{bc}^0-J_{m, bc}X_*$, where
$X_*$ is a solution of nonlinear equation
\begin{equation}\label{eq521H}
X=B\Gamma_{m, bc}X-(\Gamma_{m, bc}X)(J_{m, bc}B)-(\Gamma_{m, bc}X)J_{m, bc}(B\Gamma_{m, bc}X)+B.
\end{equation}
In this equation the operator $B$ is defined in (\ref{v0thH}) and $l$ satisfies Theorem~\ref{thsimilar-3H}.
\end{theorem}

\begin{proof}
Follows from Theorems~\ref{thsimilar-2} and \ref{thsimilar-3H}.
\end{proof}

We end this section with some additional properties of the operator $B$ that will be useful for our further consideration
in the next section.

From (\ref{v0thH}) it follows that
\begin{flalign}\label{newB}
B &=\big(\sum_{j=0}^\infty(-1)^j(\Gamma_{l, bc}Q)^j\big)(Q\Gamma_{l, bc}Q-(\Gamma_{l, bc}Q)J_{l, bc}Q) =
Q\Gamma_{l, bc}Q \nonumber\\
&- (\Gamma_{l, bc}Q)J_{l, bc}Q - (\Gamma_{l, bc}Q)(I+\Gamma_{l, bc}Q)^{-1}
(Q\Gamma_{l, bc}Q - (\Gamma_{l, bc}Q)J_{l, bc}Q).
\end{flalign}

Let $bc\in\{per, ap\}$. Applying the operator $\mathbb{P}_n$ to (\ref{newB})
on the right side and on the left side, and using the equality
$(J_{l, bc}Q)\mathbb{P}_n=\mathbb{P}_n(J_{l, bc}Q)=\mathbb{P}_n(J_{l, bc}Q)\mathbb{P}_n$, we get
\begin{flalign}
&B\mathbb{P}_n - \mathbb{P}_nB\mathbb{P}_n=(Q\Gamma_{l, bc}Q)\mathbb{P}_n - \mathbb{P}_n(Q\Gamma_{l, bc}Q)\mathbb{P}_n-
(\Gamma_{l, bc}Q)\mathbb{P}_nQ\mathbb{P}_n \nonumber\\
&+ (\mathbb{P}_n(\Gamma_{l, bc}Q) - \Gamma_{l, bc}Q)(I+\Gamma_{l, bc}Q)^{-1}
((Q\Gamma_{l, bc}Q)\mathbb{P}_n-(\Gamma_{l, bc}Q)\mathbb{P}_nQ\mathbb{P}_n)\label{statbpn}
\end{flalign}
and
\begin{flalign}
&\mathbb{P}_nB - \mathbb{P}_nB\mathbb{P}_n=\mathbb{P}_n(Q\Gamma_{l, bc}Q)-\mathbb{P}_n(Q\Gamma_{l, bc}Q)\mathbb{P}_n\nonumber\\
&+ \mathbb{P}_n(\Gamma_{l, bc}Q)(I+\Gamma_{l, bc}Q)^{-1}(-Q\Gamma_{l, bc}Q + (\Gamma_{l, bc}Q)\mathbb{P}_nQ\mathbb{P}_n +
(Q\Gamma_{l, bc}Q)\mathbb{P}_n-(\Gamma_{l, bc}Q)\mathbb{P}_nQ\mathbb{P}_n)\label{statpnb}.
\end{flalign}
Then the inequality $\|\Gamma_{l, bc}Q\|_2\leq 1/2$ implies that
\begin{flalign}\label{bpn-pnbpnH}
&\|B\mathbb{P}_n - \mathbb{P}_nB\mathbb{P}_n\|_2\leq \|(Q\Gamma_{l, bc}Q)\mathbb{P}_n\|_2 +
\|\mathbb{P}_n(Q\Gamma_{l, bc}Q)\mathbb{P}_n\|_2 \nonumber\\
&+ \|(\Gamma_{l, bc}Q)\mathbb{P}_n\|_2\|\mathbb{P}_nQ\mathbb{P}_n\|_2 + \frac{\|\mathbb{P}_n(\Gamma_{l, bc}Q)\|_2 +
\|\Gamma_{l, bc}Q\|_2}{1-\|\Gamma_{l, bc}Q\|_2}\big(\|(Q\Gamma_{l, bc}Q)\mathbb{P}_n\|_2 \nonumber\\
&+ \|(\Gamma_{l, bc}Q)\mathbb{P}_n\|_2\|\mathbb{P}_nQ\mathbb{P}_n\|_2\big)\leq
2\|(Q\Gamma_{l, bc}Q)\mathbb{P}_n\|_2 + \|(\Gamma_{l, bc}Q)\mathbb{P}_n\|_2\|\mathbb{P}_nQ\mathbb{P}_n\|_2 \nonumber\\
&+ (2\|\mathbb{P}_n(\Gamma_{l, bc}Q)\|_2 + 1)(\|(Q\Gamma_{l, bc}Q)\mathbb{P}_n\|_2 +
\|(\Gamma_{l, bc}Q)\mathbb{P}_n\|_2\|\mathbb{P}_nQ\mathbb{P}_n\|_2)
\end{flalign}
and
\begin{flalign}\label{pnb-pnbpnH}
&\|\mathbb{P}_nB - \mathbb{P}_nB\mathbb{P}_n\|_2\leq \|\mathbb{P}_n(Q\Gamma_{l, bc}Q)\|_2 +
\|\mathbb{P}_n(Q\Gamma_{l, bc}Q)\mathbb{P}_n\|_2 \nonumber\\
&+ \frac{\|\mathbb{P}_n(\Gamma_{l, bc}Q)\|_2}{1-\|\Gamma_{l, bc}Q\|_2}\big(\|Q\Gamma_{l, bc}Q\|_2+
\|(Q\Gamma_{l, bc}Q)\mathbb{P}_n\|_2 + 2\|(\Gamma_{l, bc}Q)\mathbb{P}_n\|_2\|\mathbb{P}_nQ\mathbb{P}_n\|_2\big)
\leq 2\|\mathbb{P}_n(Q\Gamma_{l, bc}Q)\|_2 \nonumber\\
&+ 2\|\mathbb{P}_n(\Gamma_{l, bc}Q)\|_2(\|Q\Gamma_{l, bc}Q\|_2 +
\|(Q\Gamma_{l, bc}Q)\mathbb{P}_n\|_2 + 2\|(\Gamma_{l, bc}Q)\mathbb{P}_n\|_2\|\mathbb{P}_nQ\mathbb{P}_n\|_2).
\end{flalign}

In case $bc=dir$, by similar arguments, we obtain (\ref{statbpn})~--~(\ref{pnb-pnbpnH}) with the projections $P_{n, dir}$
instead of~$\mathbb{P}_n$.

By using (\ref{pngvH}), (\ref{vgvpnH}), (\ref{pnqgq}) and (\ref{qgq}) in (\ref{bpn-pnbpnH}) and (\ref{pnb-pnbpnH})
for $bc\in\{per, ap\}$ and (\ref{pndirgqH}), (\ref{qgdirqpn}), and (\ref{pnggdirq}) for $bc=dir$, we have
\begin{flalign}
&\|B\mathbb{P}_n-\mathbb{P}_nB\mathbb{P}_n\|_2\leq \frac{C}{(2n+\theta)^{2k-2}}\alpha(2n+\theta), \quad
\|\mathbb{P}_nB - \mathbb{P}_nB\mathbb{P}_n\|_2\leq C, \label{bpn}\\
&\|BP_{n, dir}-P_{n, dir}BP_{n, dir}\|_2\leq C\beta(2n)/n^{2k-2}, \label{bpndir}\\
&\|P_{n, dir}B - P_{n, dir}BP_{n, dir}\|_2\leq C \label{pndirb},
\end{flalign}
where the sequences $\alpha$ and $\beta$ are defined in (\ref{alpha}) and (\ref{beta}), respectively.

\section{Proofs of the main results}\label{sec4}

In this section we prove our main results.

\textit{Proof of Theorem~\ref{thasympt0and1H}.} Using Theorem~\ref{thsimilar-3H}, we transform the operator $L_{bc}$,
$bc\in\{per, ap\}$, to the operator $\widetilde{L}_{bc}^0-B$, where $B$ is defined by (\ref{v0thH}) and has the form
(\ref{newB}). Since $B$ belongs to $\mathfrak{S}_2(\mathcal{H})$, we can apply the results from Section~\ref{sec2} to the
operator $\widetilde{L}_{bc}^0-B$. By Theorem~\ref{thsimilar4H}, there exists $m\in\mathbb{Z}_+$, $m\geq l+1$, such that
the operator $L_{bc}$ is similar to the operator $\widetilde{L}_{bc}^0-J_{m, bc}X_*$, where $X_*$ is a solution of the equation
(\ref{eq521H}). Since the operator $J_{m, bc}X_*$ is bounded and $\widetilde{L}_{bc}^0$ is a normal operator with discrete spectrum,
$\widetilde{L}_{bc}^0-J_{m, bc}X_*$ and, consequently, $L_{bc}$ are operators with discrete spectrum.
Moreover,
\begin{flalign}\label{sigmamsigmanH}
\sigma(L_{bc})=\sigma(\widetilde{L}_{bc}^0 - J_{m, bc}X_*) &= \sigma(A_{(m)})\cup \Big(\bigcup_{n\geq m+1}
\Big(\lambda_n-q_0-\sigma(A_n)\Big)\Big)\nonumber\\
&= \sigma_{(m)}\cup \Big(\bigcup_{n\geq m+1}\Big(\lambda_n-q_0-\sigma_n\Big)\Big),
\end{flalign}
where $A_n$ is the restriction of the operator $\mathbb{P}_nX_*\mathbb{P}_n$ to the subspace $\mathrm{Im}\,\mathbb{P}_n$ and $A_{(m)}$
is the restriction of operator $\widetilde{L}_{bc}^0-\mathbb{P}_{(m)}X_*\mathbb{P}_{(m)}$ to the subspace
$\mathrm{Im}\,\mathbb{P}_{(m)}$. Note that the sets $\sigma_{(m)}$ and $\sigma_n$ are mutually disjoint and
$\lambda_n=\pi^{2k}(2n+\theta)^{2k}/\omega^{2k}$.

To obtain the asymptotic formulas for the eigenvalues of the operator $L_{bc}$, we now describe the sets $\sigma_n$, $n\geq m+1$.
The operators $A_n$, $n\geq m+1$, are represented as
\[
A_n=B_n+C_n + D_n, \quad n\geq m+1,
\]
where $B_n$, $C_n$ and $D_n$ are the restrictions of the operators $\mathbb{P}_n(J_{m, bc}Q)\mathbb{P}_n=\mathbb{P}_nQ\mathbb{P}_n$,
$\mathbb{P}_n(B-J_{m, bc}Q)\mathbb{P}_n$ and $\mathbb{P}_n(X_*-B)\mathbb{P}_n$ to the subspace $\mathrm{Im}\,\mathbb{P}_n$,
respectively. Then we have
\begin{equation}\label{bn0+cn0H}
\mathcal{A}_n=\mathcal{B}_n + \mathcal{C}_n +\mathcal{D}_n,
\end{equation}
where
\[
\mathcal{B}_n=
\begin{pmatrix}
0 & q_{-2n-\theta} \\
q_{2n+\theta} & 0
\end{pmatrix}, \quad
\mathcal{C}_n=
\begin{pmatrix}
(Q\Gamma_{m, bc}Qe_{-n-\theta}, e_{-n-\theta}) & (Q\Gamma_{m, bc}Qe_n, e_{-n-\theta}) \\
(Q\Gamma_{m, bc}Qe_{-n-\theta}, e_n) & (Q\Gamma_{m, bc}Qe_n, e_n)
\end{pmatrix}
\]
and $\mathcal{D}_n$ is the matrix of the operator $D_n$. By $c_{ij}$, $i, j=1, 2$, we denote the elements of matrix
$\mathcal{C}_n$ and note that
\begin{flalign}
&c_{11}=c_{22}=\frac{\omega^{2k}}{\pi^{2k}}\sum\limits_{\substack{j\ne n \\ j\ne -n-\theta}}
\frac{q_{-n-j-\theta}q_{n+j+\theta}}{(2j+\theta)^{2k}-(2n+\theta)^{2k}}, \label{c11}\\
&c_{12}=(Q\Gamma_{m, bc}Qe_n, e_{-n-\theta})=\frac{\omega^{2k}}{\pi^{2k}}\sum\limits_{\substack{j\ne n \\ j\ne -n-\theta}}
\frac{q_{-n-j-\theta}q_{j-n}}{(2j+\theta)^{2k}-(2n+\theta)^{2k}}, \label{c12}\\
&c_{21}=(Q\Gamma_{m, bc}Qe_{-n-\theta}, e_n)=\frac{\omega^{2k}}{\pi^{2k}}\sum\limits_{\substack{j\ne n \\ j\ne -n-\theta}}
\frac{q_{n-j}q_{n+j+\theta}}{(2j+\theta)^{2k}-(2n+\theta)^{2k}}\label{c21}.
\end{flalign}

To continue, we need the following notation. For sequences of complex numbers $a_n$, $b_n$, $n\geq 1$, with $a_nb_n\ne 0$,
we define the following matrices
\[
U_n^{-1}=
\begin{pmatrix}
1 & 1 \\
-\sqrt{b_n/a_n} & \sqrt{b_n/a_n}
\end{pmatrix},
\quad U_n=
\begin{pmatrix}
1/2 & -\sqrt{a_n}/(2\sqrt{b_n}) \\
1/2 & \sqrt{a_n}/(2\sqrt{b_n})
\end{pmatrix}
\]
and note that
\[
U_n
\begin{pmatrix}
0 & a_n \\
b_n & 0
\end{pmatrix}
U_n^{-1}=
\begin{pmatrix}
-\sqrt{a_nb_n} & 0 \\
0 & \sqrt{a_nb_n}
\end{pmatrix},
\quad n\geq 1.
\]

Put $a_n=q_{-2n-\theta}+c_{12}$ and $b_n=q_{2n+\theta}+c_{21}$. Multiplying the both parts of (\ref{bn0+cn0H})
by $U_n$ from the left and by $U_n^{-1}$ from the right, we have
\begin{equation}\label{newa}
\mathcal{A}_n=
\begin{pmatrix}
c_{11} & 0 \\
0 & c_{22}
\end{pmatrix}
+
\begin{pmatrix}
-\sqrt{a_nb_n} & 0 \\
0 & \sqrt{a_nb_n}
\end{pmatrix}
+ U_n\mathcal{D}_nU_n^{-1}.
\end{equation}
Now we estimate the last term in (\ref{newa}). Using Theorem~\ref{th4-asymptotic} and formulas (\ref{cnpn}) and (\ref{bpn}),
we get
\begin{flalign}\label{pnx-bpnH}
\|U_n&\mathcal{D}_nU_n^{-1}\|_2\leq \|U_n\|_2\|U_n^{-1}\|_2\|\mathcal{D}_n\|_2\leq
C\|\mathbb{P}_n(X_*-B)\mathbb{P}_n\|_2 \nonumber\\
&\leq \frac{\omega^{2k}C}{2\pi^{2k}n(2n+\theta)^{2k-2}}
\|\mathbb{P}_nB-\mathbb{P}_nB\mathbb{P}_n\|_2\|B\mathbb{P}_n-\mathbb{P}_nB\mathbb{P}_n\|_2\leq
\frac{C\alpha(2n+\theta)}{n^{4k-3}}, \quad n\geq n_0,
\end{flalign}
where $n_0=\max\{m+1, (3\|B\|_2\omega^{2k}/(2\pi^{2k}))^{1/(2k-1)}\}$ and $\alpha$ is defined by (\ref{alpha}).

Next, arguing as in Lemma~\ref{lhvgv-gsH} (see the proof of (\ref{vgvpnH})), we obtain
\[
|c_{11}| = |(Q\Gamma_{bc}Qe_{-n-\theta}, e_{-n-\theta})| = \frac{\omega^{2k}}{\pi^{2k}}\bigg|\sum\limits_{\substack{j\ne n \\
j\ne -n-\theta}}\frac{q_{-n-j}q_{n+j}}{(2j+\theta)^{2k}-(2n+\theta)^{2k}}\bigg| \leq
\frac{\omega^{2k}\|q\|_2\alpha(2n+\theta)}{2\sqrt{3}\pi^{2k-1}(2n+\theta)^{2k-2}}.
\]

From (\ref{bn0+cn0H}), (\ref{newa}) and (\ref{pnx-bpnH}), we have
\begin{flalign}\label{mu-mu}
\Big|\widetilde{\lambda}_n^\mp&-\lambda_n + q_0 + c_{11}\pm \sqrt{a_nb_n}\Big|\leq
C\alpha(2n+\theta)/n^{4k-3},
\end{flalign}
where $a_n=q_{-2n-\theta}+c_{12}$ and $b_n=q_{2n+\theta}+c_{21}$. By using the above estimates, this implies the desired
asymptotic (\ref{lambdatheta0and1*H}) and completes the proof.

\textit{Proof of Theorem~\ref{thasympvarH}}. Since $q$ is a function of bounded variation, its Fourier coefficients
$q_n$, $n\in\mathbb{Z}$, satisfy the following condition (see \cite[Ch.~2, Theorem~4.12]{<Zigmund>}):
\[
|q_n|\leq C/(|n|+1), \quad n\in\mathbb{Z}.
\]
From this and (\ref{alpha}) it then follows that
\begin{equation}\label{estalpha}
\alpha(2n+\theta)\leq C/(|n|+1), \quad n\in\mathbb{Z}_+.
\end{equation}
To complete the proof it remains to apply Theorem~\ref{thasympt0and1H}.

\begin{remark} Evidently, Theorem~\ref{thasympvarH} holds for smooth potential $q$, i.e. for $q\in C^k[0, \omega]$, $k\ge 1$.
Moreover, it improves a similar result of H.~Menken (see \cite[Theorem~3.1]{<Menken>}), since the asymptotic in this theorem
contains the second term and, consequently, the remainder term in more precise form.
\end{remark}

\textit{Proof of Theorem~\ref{thasympt(0,1)H}}. Using Theorem~\ref{thsimilar-3H}, we transform the operator $L_{dir}$
to the operator $\widetilde{L}_{dir}^0-B$, where $B$ is defined by (\ref{v0thH}) and has the form (\ref{newB}). By
Theorem~\ref{thsimilar4H}, there exists $m\in\mathbb{N}$, $m\geq l+1$, such that the operator $L_{dir}$ is similar to
$\widetilde{L}_{dir}^0-J_{m, dir}X_*$, where $X_*$ is a solution of the equation (\ref{eq521H}).
Similarly to the proof of Theorem~\ref{thasympt0and1H} we conclude that the spectrum of these two operators is descrete.
Moreover, since projections $P_{n, dir}$, $n\geq m+1$, have rang one, the spectrum has the following form
\[
\sigma(L_{dir})=\sigma(\widetilde{L}_{dir}^0-J_{m, dir}X_*)=\sigma_{(m)}\cup \big(\bigcup_{n\geq m+1}
\big\{\sigma(\widetilde{L}_{dir}^0)-(X_*e_{n, dir}, e_{n, dir})\big\}\big),
\]
where $\sigma_{(m)}$ is the restriction of the operator $\widetilde{L}_{dir}^0-J_{m, dir}X_*$ to the subspace $\mathrm{Im}\,P_{(m)}$.

Now we obtain the asymptotic behavior of the eigenvalues $\widetilde{\lambda}_{n, dir}$, $n\geq m+1$, of the operator $L_{dir}$.
From (\ref{newB}) and (\ref{qij}) it follows that
\[
(X_*e_{n, dir}, e_{n, dir}) = (Be_{n, dir}, e_{n, dir}) +
((X_*-B)e_{n, dir}, e_{n, dir}) = (Q\Gamma_{m, dir}Qe_{n, dir}, e_{n, dir}) +\eta_{dir}(n)
\]
and
\[
(Q\Gamma_{m, dir}Qe_{n, dir}, e_{n, dir}) = \frac{\omega^{2k}}{2\pi^{2k}}\sum\limits_{\substack{j=1 \\ j\ne n}}^\infty
\frac{(\widetilde{q}_{|n-j|}-\widetilde{q}_{n+j})^2}{j^{2k}-n^{2k}},
\]
respectively.
Using Theorem~\ref{th3} and formulas (\ref{etapnbpn}), (\ref{bpndir}), and
(\ref{pndirb}), this implies that
\begin{flalign*}
|\eta_{dir}(n)| &= |((X_*-B)e_{n, dir}, e_{n, dir})|\leq \|P_{n, dir}(X_*-B)P_{n, dir}\|_2 \\
&\leq \frac{2\omega^{2k}}{\pi^{2k}n^{2k-1}}
\|P_{n, dir}B-P_{n, dir}BP_{n, dir}\|_2\|BP_{n, dir}-P_{n, dir}BP_{n, dir}\|_2\leq \frac{C\beta(2n)}{n^{4k-3}}, \quad n\geq n_1,
\end{flalign*}
where $n_1=\{m+1, (6\|B\|_2\omega^{2k}\pi^{-2k})^{1/(2k-1)}\}$ and $\beta$ is defined by (\ref{beta}).

Recall that the eigenvalues $\widehat{\lambda}_{n, dir}$ of the operator $\widetilde{L}_{dir}^0$ have the following form
\[
\widehat{\lambda}_{n, dir} = \lambda_{n, dir}-(Qe_{n, dir}, e_{n, dir}) = \Big(\frac{\pi n}{\omega}\Big)^{2k} -
\frac{1}{\omega}\int_0^\omega q(t)\,dt + \frac{1}{\omega}\int_0^\omega q(t)\cos\frac{2\pi n}{\omega}t\,dt.
\]

By using the above relations we get the asymptotic representation (\ref{lambdantheta01H}).

It remains to establish the asymptotic (\ref{qdirbondvar}) whenever $q$ is a function of bounded variation.
But in this case the proof is similar to the proof of Theorem~\ref{thasympvarH} and we omit it. This completes the proof.

\begin{remark}
Note that Theorem~\ref{thasympt(0,1)H} improves \cite[Theorem~1]{Ahmerova}, since the asymptotic in Theorem~\ref{thasympt(0,1)H}
contains the second term and the remainder term in more precise form.
\end{remark}

\textit{Proof of Theorem~\ref{thPPH}.} By Theorem~\ref{thsimilar-3H}, the operator $L_{bc}$ is similar to the operator
$\widetilde{L}_{bc}^0-B$ and
\begin{equation}\label{ulu}
L_{bc}(I+\Gamma_{l, bc}Q)=(I+\Gamma_{l, bc}Q)(\widetilde{L}_{bc}^0-B),
\end{equation}
where $B$ is defined by (\ref{v0thH}) and has the form (\ref{newB}). Taking into account Theorem~\ref{thsimilar4H} and
conditions (\ref{9-estimate}), (\ref{10-estimate}), there exists a number $m\in\mathbb{J}$, $m\geq l+1$, such that the operator
$L_{bc}$ is similar to $\widetilde{L}_{bc}^0-J_{m, bc}X_*$, where $X_*$ is a solution of the equation~(\ref{eq521H}). Then
\begin{equation}\label{ulu2}
(\widetilde{L}_{bc}^0-B)(I+\Gamma_{m, bc}X_*)=(I+\Gamma_{m, bc}X_*)(\widetilde{L}_{bc}^0-J_{m, bc}X_*).
\end{equation}
Denote by $\mathcal{U}$ and $\mathcal{V}$ the operators $\Gamma_{l, bc}Q$ and
$\Gamma_{m, bc}X_*$, respectively. Then the equalities (\ref{ulu}) and (\ref{ulu2}) imply that
\begin{equation}\label{lbc}
L_{bc}= (I+\mathcal{U})(I+\mathcal{V})(\widetilde{L}_{bc}^0 -J_{m, bc}X_*)(I+\mathcal{V})^{-1}(I+\mathcal{U})^{-1}.
\end{equation}

Let $\Omega$ be a subset of $\{m+1, m+2, \dots\}$. Put $\Delta=\Delta(\Omega)=\cup_{n\in\Omega}\{\lambda_n\}$ and
$\widetilde{\Delta}=\widetilde{\Delta}(\Omega)=\cup_{n\in\Omega}\sigma_n$, where $\sigma_n$ is defined by~(\ref{sigmamsigmanH}).
If $bc=dir$, then $\Delta=\Delta(\Omega)=\cup_{n\in\Omega}\{\lambda_{n, dir}\}$ and
$\widetilde{\Delta}=\widetilde{\Delta}(\Omega)=\cup_{n\in\Omega}\sigma_n$, where
$\sigma_n=\{(X_*e_{n, dir}, e_{n, dir})\}$, $n\in\Omega$.
For these sets we define the projections $P(\widetilde{\Delta}, L_{bc})$ and $P(\Delta, L_{bc}^0)$
(see Section~\ref{sec1}). Applying \cite[Lemma~1]{<2011>} and using (\ref{lbc}), we have
\begin{equation}\label{pupH}
P(\widetilde{\Delta}, L_{bc})=(I+\mathcal{U})(I+\mathcal{V})P(\Delta, L_{bc}^0)(I+\mathcal{V})^{-1}(I+\mathcal{U})^{-1}.
\end{equation}

Next, from Theorems~\ref{thsimilar-3H} and \ref{thsimilar4H}  it follows that
\[
\|\mathcal{U}\|_2=\|\Gamma_{l, bc}Q\|_2\leq 1/2, \quad \|\mathcal{V}\|_2=\|\Gamma_{m, bc}X_*\|_2\leq 1/2.
\]
Using these inequalities and representation (\ref{pupH}), we obtain
\begin{flalign*}
P(\widetilde{\Delta}, L_{bc}) &- P(\Delta, L_{bc}^0) =
\mathcal{U}P(\Delta, L_{bc}^0) + \mathcal{V}P(\Delta, L_{bc}^0) + \mathcal{U}\mathcal{V}P(\Delta, L_{bc}^0) \\
&+ \Big(\mathcal{U}P(\Delta, L_{bc}^0) + \mathcal{V}P(\Delta, L_{bc}^0) +
\mathcal{U}\mathcal{V}P(\Delta, L_{bc}^0)\Big)\Big(\sum_{j=1}^\infty \mathcal{V}^j
+ \sum_{j=1}^\infty\mathcal{U}^j + \sum_{j=1}^\infty\mathcal{V}^j\sum_{j=1}^\infty\mathcal{U}^j\Big).
\end{flalign*}
Then
\begin{flalign*}
&\|P(\widetilde{\Delta}, L_{bc}) - P(\Delta, L_{bc}^0)\|_2\leq
\|\mathcal{U}P(\Delta, L_{bc}^0)\|_2 + \|\mathcal{V}P(\Delta, L_{bc}^0)\|_2 \\
&+ \|\mathcal{U}\|_2\|\mathcal{V}P(\Delta, L_{bc}^0)\|_2 +
\Big(\|\mathcal{U}P(\Delta, L_{bc}^0)\|_2 + \|\mathcal{V}P(\Delta, L_{bc}^0)\|_2 \\
&+ \|\mathcal{U}\|_2\|\mathcal{V}P(\Delta, L_{bc}^0)\|_2\Big)\Big(\sum_{j=1}^\infty\|\mathcal{V}\|_2^j
+\sum_{j=1}^\infty\|\mathcal{U}\|_2^j +
\sum_{j=1}^\infty\|\mathcal{V}\|_2^j\sum_{j=1}^\infty\|\mathcal{U}\|_2^j\Big) \\
&\leq 4\|\mathcal{U}P(\Delta, L_{bc}^0)\|_2 + 6\|\mathcal{V}P(\Delta, L_{bc}^0)\|_2.
\end{flalign*}

To continue, we estimate $\|\mathcal{U}P(\Delta, L_{bc}^0)\|_2$ and $\|\mathcal{V}P(\Delta, L_{bc}^0)\|_2$. Let $bc\in\{per, ap\}$
and $d(\Omega)=\min\limits_{n\in \Omega}n$. Using~(\ref{5-gamma}), we get
\[
\|\mathcal{U}P(\Delta, L_{bc}^0)\|_2^2=\sum_{n\geq d(\Omega)}\|(\Gamma_{l, bc}Q)\mathbb{P}_n\|_2^2\leq
\frac{\omega^{4k}\|q\|^2_2}{2\pi^{4k}}\sum_{n\geq d(\Omega)}\frac{1}{(2n+\theta)^{4k-4}(2n+1)}\leq \frac{C}{d^{4k-3}(\Omega)}.
\]
By similar arguments, the same estimation holds for $\|\mathcal{V}P(\Delta, L_{bc}^0)\|_2$.  Therefore,
\[
\|P(\widetilde{\Delta}, L_{bc})-P(\Delta, L_{bc}^0)\|_2\leq C/d^{2k-\frac{3}{2}}(\Omega),
\]
which completes the proof for $bc\in\{per, ap\}$.

For $bc=dir$ the proof is similar and we omit it. This proves Theorem~\ref{thPPH}.

\textit{Proof of Theorem~\ref{thsemigroup1H}}. By Theorem~\ref{thsimilar4H}, the operator $L_{bc}$, $bc\in\{per, ap, dir\}$,
is similar to $\widetilde{L}_{bc}^0-J_{m, bc}X_*$. Since $J_{m, bc}X_*$ is bounded and $-\widetilde{L}_{bc}^0$ is
sectorial (see \cite[Theorem~1.3.2]{Henry}), the operator $-\widetilde{L}_{bc}^0+J_{m, bc}X_*$ is sectorial
(see~\cite[Sec.~5]{PolyakovAA}). Hence, the operator $-L_{bc}$ is sectorial and generates an analytic semigroup of operators
$T: \mathbb{R}_+\to \mathrm{End}\,\mathcal{H}$. By Theorem~\ref{thsimilar4H}, there exists $m\in\mathbb{J}$ such that
this semigroup is similar to a semigroup $\widetilde{T}: \mathbb{R}_+\to\mathrm{End}\,\mathcal{H}$ of the form
\[
\widetilde{T}(t)=T_{(m)}(t)\oplus T^{(m)}(t), \quad t\in\mathbb{R}_+,
\]
acting in $\mathcal{H}=\mathcal{H}_{(m)}\oplus\mathcal{H}^{(m)}$. Here
$\mathcal{H}_{(m)}=\mathrm{Im}\,\mathbb{P}_{(m)}$, $\mathcal{H}^{(m)}=\mathrm{Im}\,(I-\mathbb{P}_{(m)})$ for $bc\in\{per, ap\}$ and
$\mathcal{H}_{(m)}=\mathrm{Im}\,P_{(m)}$, $\mathcal{H}^{(m)}=\mathrm{Im}\,(I-P_{(m)})$ for $bc=dir$.

If $bc=dir$, then by \cite{Henry} the semigroup $T^{(m)}: \mathbb{R}_+\to \mathrm{End}\,\mathcal{H}^{(m)}$ has the
representation
\[
T^{(m)}(t)x=\sum_{s\geq m+1}e^{-\widetilde{\lambda}_{s, dir}t}P_{s, dir}x, \quad x\in\mathcal{H},
\]
where $\widetilde{\lambda}_{s, dir}$ is defined in (\ref{lambdantheta01H}).

Let $bc\in\{per, ap\}$. To continue, we need the following result (see \cite[Ch.~1, Sec.~6, problem~2]{<Kirillov>}).
\begin{lemma}\label{lhsemigroup}
Let $\mathcal{A}$ be a matrix of the form
$
\begin{pmatrix}
a & b \\
c & d
\end{pmatrix}
$.
Then
\[
e^{\mathcal{A}t}=e^{\frac{a+d}{2}}\bigg(\mathrm{ch}(\rho t)
\begin{pmatrix}
I & 0 \\
0 & I
\end{pmatrix}
+\frac{\mathrm{sh}(\rho t)}{\rho}
\begin{pmatrix}
\frac{a-d}{2} & b \\
c & \frac{d-a}{2}
\end{pmatrix}
\bigg),
\]
where $\rho=\sqrt{(a-d)^2/2+bc}$.
\end{lemma}

Put $\mathcal{A}=\pi^{2k}(2n+\theta)^{2k}I_n/\omega^{2k}-\mathcal{A}_n$, where $\mathcal{A}_n$ is
the matrix of the restriction of the operator $J_{m, bc}(\mathbb{P}_nX_*\mathbb{P}_n)$ to the subspace $\mathcal{H}_n$ in the basis
$e_{-n-\theta}$, $e_n$. Applying Lemma~\ref{lhsemigroup} and the representation (\ref{bn0+cn0H}), we get
\begin{flalign}\label{tmrepresentH}
\widetilde{T}^{(m)}(t)x &= \sum_{s\geq m+1} e^{-t\lambda_n + q_0t + tc_{11}}e^{a_st}\Big(\mathrm{ch}b_st +
\frac{\mathrm{sh}b_st}{b_s}B_s\Big)\mathbb{P}_sx, \quad x\in\mathcal{H},
\end{flalign}
where $\lambda_n=\pi^{2k}(2n+\theta)^{2k}/\omega^{2k}$, $b_s=\sqrt{(q_{-2s-\theta}+ c_{12})(q_{2s+\theta} + c_{21})+\varepsilon_s}$
with $c_{11}$, $c_{12}$, $c_{21}$, defined by formulas (\ref{c11})~--~(\ref{c21}), and some constants $\varepsilon_s$, and
$B_s\in\mathrm{End}\,\mathcal{H}$ with $B_sx=0$ for all $x\in\mathrm{Ker}\,\mathbb{P}_s$, $s\geq m+1$. Let us represent the
operators $B_s$ in the form $B_s=B_s^0+B_s'$, where the operators $B_s^0$ are defined by the following relations:
$B_s^0e_{-s-\theta}=(q_{2s+\theta} + c_{21})e_s$,
$B_s^0e_s=(q_{-2s-\theta}+c_{12})e_{-s-\theta}$, and $B_s^0x=0$ for all $x\in\mathrm{Ker}\,\mathbb{P}_s$.
Then, by the proof of Theorem~\ref{thasympt0and1H} (see inequality (\ref{mu-mu})), the sequences of complex numbers $a_s$ and
$\varepsilon_s$, $s\geq m+1$, and the sequence
of operators $B_s'$, $s\geq m+1$, satisfy the estimate:
\[
\max\{|a_s|, |\varepsilon_s|, \|B_s'\|_2\}\leq C\alpha(2s+\theta)/s^{4k-3}, \quad s\geq m+1,
\]
which completes the proof.

\end{document}